\documentclass[12pt,reqno]{amsart}
\usepackage[a4paper]{geometry}
\usepackage{amssymb,amsmath,amsthm,amsbsy,enumerate,verbatim,bbm}
\usepackage{graphicx}
\DeclareMathOperator{\Ker}{Ker}
\DeclareMathOperator{\Tr}{Tr}
\DeclareMathOperator{\Ran}{Ran}
\DeclareMathOperator{\dist}{dist}
\DeclareMathOperator{\supp}{supp}
\DeclareMathOperator{\conv}{conv}
\DeclareMathOperator{\essran}{ess\, ran}
\DeclareMathOperator{\meas}{meas}
\DeclareMathOperator{\essinf}{ess\, inf}
\DeclareMathOperator{\esssup}{ess\, sup}

\renewcommand{\Im}{\operatorname{Im}}

\newcommand{\abs}[1]{\lvert#1\rvert}
\newcommand{\Abs}[1]{\left\lvert#1\right\rvert}
\newcommand{\norm}[1]{\lVert#1\rVert}
\newcommand{\jap}[1]{\langle#1\rangle}

\newcommand{\bbC}{{\mathbb C}}
\newcommand{\bbE}{{\mathbf E}}
\newcommand{\bbP}{{\mathbf P}}
\newcommand{\bbR}{{\mathbb R}}
\newcommand{\bbZ}{{\mathbb Z}}
\newcommand{\bbT}{{\mathbb T}}
\newcommand{\bbV}{{\mathbf V}}

\newcommand{\calH}{\mathcal{H}}
\newcommand{\calN}{\mathcal{N}}
\newcommand{\calF}{\mathcal{F}}
\newcommand{\calP}{\mathcal{P}}

\newcommand{\myQ}{{\mathcal Q}}

\numberwithin{equation}{section}

\theoremstyle{plain}
\newtheorem{theorem}{\bf Theorem}[section]
\newtheorem*{theorem*}{Theorem}
\newtheorem{lemma}[theorem]{\bf Lemma}
\newtheorem{proposition}[theorem]{\bf Proposition}
\newtheorem*{proposition*}{\bf Proposition}
\newtheorem{assumption}[theorem]{\bf Assumption}

\theoremstyle{definition}

\theoremstyle{remark}
\newtheorem*{remark*}{\bf Remark}
\newtheorem{remark}[theorem]{\bf Remark}
\newtheorem{example}[theorem]{\bf Example}

\newcommand{\eps}{\varepsilon}

\newcommand{\dd}{\mathrm d}
\newcommand{\ii}{\mathrm i}
\newcommand{\ee}{\mathrm e}

\newcommand{\kappamax}{{\varkappa_{\max}}}

\makeatletter
\newcommand*\bigcdot{\mathpalette\bigcdot@{.5}}
\newcommand*\bigcdot@[2]{\mathbin{\vcenter{\hbox{\scalebox{#2}{\,\,$\m@th#1\bullet$\,\,}}}}}
\makeatother

\begin{document}

\title[Sums of projections with random coefficients]{Sums of projections with random coefficients}

\author{Leonid Pastur}
\address{Department of Mathematics, King's College London, Strand, London, WC2R~2LS, United Kingdom}
\address{B. Verkin Institute for Low Temperature Physics and Engineering, 47 Nauky Avenue, Kharkiv, 61103 Ukraine}
\email{leonid.pastur@kcl.ac.uk}

\author{Alexander Pushnitski}
\address{Department of Mathematics, King's College London, Strand, London, WC2R~2LS, United Kingdom}
\email{alexander.pushnitski@kcl.ac.uk}



\date{25 September 2025}

\begin{abstract}
We study infinite sums 
\[
{\mathcal P}_{\varkappa}=\sum_{n=-\infty}^\infty \varkappa_n \langle\bigcdot, \psi_n\rangle\psi_n
\]
of rank-one projections in a Hilbert space, where $\{\psi_n\}_{n\in\mathbb Z}$ are norm-one vectors, not necessarily orthogonal, and $\{\varkappa_n\}_{n\in\mathbb Z}$ are independent identically distributed positive random variables. Assuming that the Gram matrix $\{\langle\psi_n,\psi_m\rangle\}_{n,m\in\mathbb Z}$ defines a bounded operator on $\ell^2(\mathbb Z)$ and that its entries depend only on the difference $n-m$, we analyse ${\mathcal P}_{\varkappa}$ within the framework of spectral theory of ergodic operators. Inspired by the spectral theory of ergodic Schr\"odinger operators, we define the integrated density of states (IDS) measure $\nu_{\calP_\varkappa}$ for ${\mathcal P}_{\varkappa}$ and establish results on its continuity and absolute continuity, including Wegner-type estimates and Lifshitz tail behaviour near the spectral edges. In the asymptotic regime of nearly-orthogonal $\psi_n$, we prove the Anderson-type localisation result: the spectrum of ${\mathcal P}_{\varkappa}$ is pure point almost surely. 
\end{abstract}

\maketitle

\section{Introduction}
\label{sec.aa}

\subsection{Overview}

Let $\calH$ be a Hilbert space, with the inner product of $f$ and $g$ denoted by $\jap{f,g}$, which we assume linear in $f$ and anti-linear in $g$.  We denote by $\jap{\bigcdot,f}f$ the rank-one operator mapping $g$ to $\jap{g,f}f$. 

Let $\{\psi_n\}_{n\in\bbZ}$ be a sequence of elements in $\calH$ of norm one, \emph{not necessarily orthogonal to one another}. We discuss the spectral analysis of the sum of rank-one projections
\begin{equation}
\boxed{
\calP_\varkappa=\sum_{n=-\infty}^\infty \varkappa_n \jap{\bigcdot,\psi_n}\psi_n}
\label{eq:aa1}
\end{equation}
where $\varkappa=\{\varkappa_n\}_{n\in\bbZ}$ are independent identically distributed (i.i.d.) real random variables, whose common distribution is denoted by $\bbP_0$. 

The objects determining $\calP_\varkappa$, up to unitary equivalence, are the distribution $\bbP_0$ and the Gram matrix of $\{\psi_n\}_{n\in\bbZ}$, i.e. 
\begin{equation}
G=\{\jap{\psi_n,\psi_m}\}_{n,m\in\bbZ}.
\label{eq:aa3}
\end{equation}
Throughout the paper, our standing assumption is 

\begin{assumption}\label{ass:1}
\begin{enumerate}[\rm (a)]
\item
$\{\varkappa_n\}_{n\in\bbZ}$ are bounded and positive, separated away from zero, i.e. $\bbP_0$ is compactly supported on $(0,\infty)$. We denote 
\[
0<\varkappa_{\min}=\inf(\supp\bbP_0), \quad \varkappa_{\max}=\sup(\supp\bbP_0)<\infty. 
\]
\item
The Gram matrix $G$ is bounded as an operator on $\ell^2(\bbZ)$.
\item
The Gram matrix $G$ is \emph{invariant with respect to the shifts of indices,} i.e. 
\begin{equation}
\jap{\psi_{n+1},\psi_{m+1}}
=
\jap{\psi_n,\psi_m}, \quad n,m\in\bbZ.
\label{eq:aa2}
\end{equation}
In other words, $\jap{\psi_n,\psi_m}$ depends only on the difference $n-m$. 
\end{enumerate}
\end{assumption}

In Section~\ref{sec:aa1} we will see that Assumptions~\ref{ass:1}(a) and (b) ensure that for all $\varkappa$, the finite-rank approximations
\begin{equation}
\calP_\varkappa^{(N)}:=\sum_{n=-N}^N \varkappa_n \jap{\bigcdot,\psi_n}\psi_n
\label{eq:PN}
\end{equation}
converge strongly as $N\to\infty$ to the operator $\calP_\varkappa$ that is bounded, self-adjoint and positive semi-definite. This provides the precise definition of  $\calP_\varkappa$. 

In Section~\ref{sec:aa3} we will see that Assumption~\ref{ass:1}(c) ensures that $\calP_\varkappa$ is an \emph{ergodic family of operators} (see Appendix for the definition). This connects our model to the large body of mathematical physics literature on discrete Schr\"odinger operators with  random i.i.d. potential (referred to as the Anderson model below) and related spectral theory of self-adjoint ergodic operators, see e.g. \cite{Pa-Fi:92}. To a large extent, we are motivated by this area and most of our methods are borrowed from this literature. However, the model we suggest has more ``pure mathematics'' than ``mathematical physics'' flavour. Because of this, we believe it has interesting new features and offers new insights.

We work under Assumption~\ref{ass:1} for the entirety of the paper. We will also make some further technical assumptions as we go along. In Section~\ref{sec:aa3} (see also Appendix) we will explain that as a consequence of ergodicity, by well-known results \cite{Pa-Fi:92} the spectrum of $\calP_\varkappa$ is deterministic (i.e. is independent of $\varkappa$ for a.e. $\varkappa$). In brief, our main results are as follows:
\begin{itemize}
\item
The location of the deterministic spectrum of $\calP_\varkappa$. 
\item
Existence of the integrated density of states (IDS) measure $\nu_{\calP_\varkappa}$ for $\calP_\varkappa$ (the limit of the normalised eigenvalue counting measures of $\calP^{(N)}_\varkappa$ of \eqref{eq:PN}) and Wegner-type estimates for the density of $\nu_{\calP_\varkappa}$. 
\item
Lifshitz tails bounds for $\nu_{\calP_\varkappa}$ ($\nu_{\calP_\varkappa}$ is ``exponentially thin'') near the top and bottom edges of the spectrum. 
\item
Pure point spectrum (Anderson type localisation) in the asymptotic regime of nearly-orthogonal vectors $\{\psi_n\}_{n\in\bbZ}$ in \eqref{eq:aa1}. 
\end{itemize}

\subsection{Discussion}
\emph{Unitary invariance}: the Gram matrix $G$ completely determines the sequence $\{\psi_n\}_{n\in\bbZ}$ up to unitary equivalence. For further references, we display this statement:
\begin{remark}\label{rmk:a1}
If $\{\psi_n\}_{n\in\bbZ}$ and $\{\widetilde{\psi}_n\}_{n\in\bbZ}$ are two sequences with the same Gram matrix, then there exists a unitary operator $U$ on $\calH$ such that $\widetilde{\psi}_n=U\psi_n$ for all $n$.
\end{remark}
Since we are interested in unitary invariant properties of $\calP_\varkappa$, the choice of the Hilbert space $\calH$ and of the sequence $\{\psi_n\}_{n\in\bbZ}$ are immaterial to us, as long as the Gram matrix is fixed. However, in Section~\ref{sec:bbb} we will discuss a concrete \emph{realisation} of $\calP_\varkappa$ in $\calH=\ell^2(\bbZ)$ which will be very useful. 

\emph{Kernels:} it is of little importance for the spectral analysis of $\calP_\varkappa$ whether the sequence $\{\psi_n\}_{n\in\bbZ}$ is complete in $\calH$ (i.e. whether the span of $\{\psi_n\}_{n\in\bbZ}$ is dense). If it is not complete, we can always focus on the subspace where it is complete, and the orthogonal complement to this subspace will be in the kernel of $\calP_\varkappa$ for all $\varkappa$. 
We will come back to the discussion of kernels in Section~\ref{sec.b5}.

\emph{The positivity of $\varkappa_n$, $n\in\bbZ$} is a strong assumption that we heavily use. It is possible to extend some aspects of our analysis to $\varkappa$ of variable sign, at the expense of additional assumptions on the Gram matrix $G$. However, some statements completely fail for $\varkappa$ of variable sign; our general impression is that this is a much more complex model. 

\emph{Related literature:}
We believe the model $\calP_\varkappa$ has inherent interest. However, we are also motivated by its connection with ergodic Hankel operators, which will be discussed in \cite{PasPush2}, as well as by the discrete Schr\"odinger operator with ergodic potential, i.i.d. random potential in particular. To make the latter motivation more explicit, write $\varkappa_n=\varkappa_{\min} + \beta_n, \; n \in \mathbb{Z}$, and insert this in \eqref{eq:aa1} to obtain
\[
\calP_\varkappa=\varkappa_{\min}\calP_1+\sum_{n=-\infty}^\infty \beta_n \jap{\bigcdot,\psi_n}\psi_n.
\]
Here the first term in the r.h.s is non-random ($\calP_1$ corresponds to the sum of projections with all coefficients $=1$) and should be viewed as the unperturbed operator, while the second term can be viewed as a non-local analogue of a random potential. In the context of Schr\"odinger operators such non-local models were considered in \cite{Fi-Pa:90, HKK}. It is also worth mentioning the sample covariance matrices of statistics and random matrix theory \cite[Chapters 7 and 19]{Pa-Sch:11}, because these matrices can also be written as sums of random projections. Related classes of large random matrices can be found in \cite{Pa:23,Pa-Va:00}.

\subsection{The structure of the paper}
In Section~\ref{sec.bb} we explain the definition of basic objects in more detail and state the main results. In Section~\ref{sec:bbb} we discuss a \emph{realisation} of $\calP_\varkappa$ of \eqref{eq:aa1} in $\ell^2(\bbZ)$; in practice, we will work with this realisation. Proofs of the main results are given in Sections~\ref{sec.cc}--\ref{sec.f}. In Appendix, we collect standard background on spectral theory of ergodic operators.

\subsection{Notation}
The sequence $\varkappa=\{\varkappa_n\}_{n\in\bbZ}$ will be viewed as a random variable on the probability space $\Omega=\bbR^\bbZ$ with respect to the probability measure $\bbP$ that is the product of countably many copies of $\bbP_0$:
\begin{equation}
\bbP=\prod_{n\in\bbZ}\bbP_0.
\label{eq:P0}
\end{equation}
\emph{Almost surely} (a.s.) refers to the measure $\bbP$. We will denote by $\bbE$ the expectation with respect to the measure $\bbP$.

For a set $\Lambda\subset\bbR$, we denote by $\chi_\Lambda$ the characteristic function of $\Lambda$. If $A$ is a self-adjoint operator, then $\chi_\Lambda(A)$ is the spectral projection of $A$ corresponding to the set $\Lambda$ and $\sigma(A)$ is the spectrum of $A$. If $A$ is a Hermitian matrix, we denote by 
\begin{equation}
\calN(\lambda;A)=\Tr\chi_{(\lambda,\infty)}(A)
\label{eq:CF}
\end{equation}
the eigenvalue counting function of $A$. (Notice that this is a slightly non-standard convention, as our $\calN(\lambda;A)$ is decreasing in $\lambda$.)

For a real-valued function $\varphi\in L^\infty(\bbT)$, where $\bbT$ is the unit circle, we denote 
\[
\essran\varphi=\{a\in\bbR: \meas\{x: \abs{\varphi(x)-a}<\eps\}>0,\   \forall \eps>0\},
\]
where $\meas$ is the Lebesgue measure on $\bbT$, and 
\[
\essinf\varphi=\inf(\essran\varphi), \quad \esssup\varphi=\sup(\essran\varphi).
\]
We denote by $\delta_n$, $n\in\bbZ$, elements of the canonical basis in $\ell^2(\bbZ)$, i.e. $\delta_n\in\ell^2(\bbZ)$ has $1$ on the $n$-th position and $0$ elsewhere.  The orthogonal projection in $\ell^2(\bbZ)$ onto the subspace spanned by $\delta_n$, $-N\leq n\leq N$ is
denoted by $1_N$. 

If $X$ and $Y$ are subsets of $\bbR$, we will write
\begin{equation}
X+Y=\{x+y: x\in X, \ y\in Y\} 
\quad\text{ and }\quad
X\cdot Y=\{xy: x\in X, \ y\in Y\}.
\label{eq:sumset}
\end{equation}

\section{Main results}
\label{sec.bb}

\subsection{Boundedness of $\calP_\varkappa$}\label{sec:aa1}
To start the discussion, we consider the case when all variables $\varkappa_n=1$, $n\in\bbZ$, and define
\[
\calP_1^{(N)}:=\sum_{n=-N}^N \jap{\bigcdot,\psi_n}\psi_n.
\]
Let us address the basic question: does $\calP_1^{(N)}$ converge strongly to a bounded operator as $N\to\infty$? Let us define an auxiliary operator $A_N:\calH\to\bbC^{2N+1}$ by 
\[
A_N: x\mapsto \{\jap{x,\psi_n}\}_{\abs{n}\leq N}.
\]
Then the adjoint $A_N^*:\bbC^{2N+1}\to\calH$ is 
\[
A_N^*: \{x_n\}_{\abs{n}\leq N}\mapsto \sum_{n=-N}^N x_n \psi_n
\]
and we have 
\[
\calP_1^{(N)}=A_N^*A_N 
\quad\text{ and }\quad
A_NA_N^*=G^{(N)}=\{\jap{\psi_n,\psi_m}\}_{n,m=-N}^N 
\]
where the $(2N+1)\times(2N+1)$ matrix $G^{(N)}$ is a truncation of the Gram matrix \eqref{eq:aa3}. It follows then from Assumption~\ref{ass:1}(b) (the boundedness of $G$) that the norms of $A_N$ are uniformly bounded. Explicitly, 
\[
\norm{A_N x}^2=\sum_{n=-N}^N \abs{\jap{x,\psi_n}}^2\leq C\norm{x}^2,
\]
and so the series $\sum_n \abs{\jap{x,\psi_n}}^2$ converges. 
This implies that as $N\to\infty$, the operators $A_N$ converge strongly to a bounded operator $A:\calH\to\ell^2(\bbZ)$, 
\[
A: x\mapsto \{\jap{x,\psi_n}\}_{n\in\bbZ},
\]
and 
$\calP_1^{(N)}$ converges strongly to a bounded self-adjoint  operator $\calP_1$, with 
\begin{equation}
\calP_1=A^*A\quad\text{ and }\quad G=AA^*.
\label{eq:aa5}
\end{equation}
Coming back to the finite-rank operators $\calP_{\varkappa}^{(N)}$ of \eqref{eq:PN}, from here we see that by the boundedness of $\varkappa_n$ (Assumption~\ref{ass:1}(a)), the norms of $\calP_{\varkappa}^{(N)}$ are uniformly bounded and $\calP_{\varkappa}^{(N)}$ converge strongly to a bounded self-adjoint positive semi-definite operator $\calP_\varkappa$ as $N\to\infty$.

\subsection{Ergodicity of $\calP_\varkappa$}\label{sec:aa3}
The shift
\[
T:\{\varkappa_n\}\mapsto \{\varkappa_{n+1}\}
\]
is an ergodic map on $\Omega=\bbR^{\bbZ}$ (i.e. it is measure preserving and if a set is invariant under $T$, then this set has $\bbP$-measure $0$ or $1$). Trivially, we have
\begin{equation}
\calP_{T\varkappa}=
\sum_{n=-\infty}^\infty \varkappa_{n+1} \jap{\bigcdot,\psi_{n}}\psi_{n}=
\sum_{n=-\infty}^\infty \varkappa_n \jap{\bigcdot,\psi_{n-1}}\psi_{n-1}.
\label{eq:aa6}
\end{equation}
The shift invariance \eqref{eq:aa2} of $G$ means that the shifted sequence $\{\psi_{n-1}\}_{n\in\bbZ}$ has the same Gram matrix as the original sequence $\{\psi_{n}\}_{n\in\bbZ}$, and therefore  (see Remark~\ref{rmk:a1})
\[
\psi_{n-1}=U\psi_n, \quad n\in\bbZ
\]
for some unitary operator $U$. Coming back to \eqref{eq:aa6}, we find that 
\[
\calP_{T\varkappa}=U\calP_\varkappa U^*.
\]
This means that  $\calP_\varkappa$ is an \emph{ergodic operator}. We have collected some minimal background on ergodic operators in the Appendix. 

By standard results in the theory (see Proposition~\ref{prp:A1}), \emph{the spectrum of $\calP_\varkappa$ is non-random}, and the same applies to the components of the spectrum (continuous, absolutely continuous, singular continuous and pure point). 

\subsection{The symbol of $\calP_\varkappa$}\label{sec:aa3a}
Here we introduce the symbol of $\calP_\varkappa$, which plays an important role in our analysis. We recall that the Gram matrix $G$ is bounded on $\ell^2(\bbZ)$ (Assumption~\ref{ass:1}(b)). This implies that any row or column of $G$ is square summable, e.g. 
\begin{equation}
\sum_{n=-\infty}^\infty \abs{\jap{\psi_n,\psi_0}}^2<\infty.
\label{eq:aa7}
\end{equation}
Now let us define the function $\varphi$ on the unit circle by 
\begin{equation}
\varphi(k)=\sum_{n=-\infty}^\infty \jap{\psi_n,\psi_0}\ee^{\ii n k}, \quad -\pi\leq k\leq \pi.
\label{eq:sym}
\end{equation}
We will call $\varphi$ the \emph{symbol} of $\calP_\varkappa$ of \eqref{eq:aa1}. 
By \eqref{eq:aa7}, the symbol $\varphi$ is in $L^2(\bbT)$. Let us denote for brevity by 
\begin{equation}
\widehat{\varphi}_n=\jap{\psi_n,\psi_0}
\label{eq:fphi}
\end{equation}
the Fourier coefficients of $\varphi$. By the shift invariance \eqref{eq:aa2}, we have 
\[
\widehat{\varphi}_{-n}=\jap{\psi_{-n},\psi_0}=\jap{\psi_0,\psi_n}=\overline{\widehat{\varphi}_n},
\]
and so $\varphi$ is real-valued. Furthermore, we can write
\begin{equation}
\jap{\psi_n,\psi_m}=\jap{\psi_{n-m},\psi_0}=\widehat{\varphi}_{n-m}=G_{n-m},
\label{eq:ppg}
\end{equation}
i.e. the Gram matrix $G$ coincides with the matrix of convolution with the sequence $\widehat{\varphi}_n$. Such convolution matrices are known as \emph{Laurent operators} (or two-sided Toeplitz operators) with the symbol $\varphi$, and we will therefore adopt the notation 
\begin{equation}
G=L_\varphi=\{\widehat{\varphi}_{n-m}\}_{n,m\in\bbZ}. 
\label{eq:aa31}
\end{equation}
Trivially, by the unitary map between $L^2(\bbT)$ and $\ell^2(\bbZ)$ effected by the Fourier series expansion, the Laurent operator $L_\varphi$ is unitarily equivalent to the operator of multiplication by the symbol $\varphi$ in $L^2(\bbT)$. In particular,  the spectrum of $L_\varphi=G$ coincides with the essential range of $\varphi$:
\[
\sigma(L_\varphi)=\essran\varphi.
\]
Since the Gram matrix is always positive semi-definite, we find that the symbol $\varphi$ is automatically non-negative a.e. on $\bbT$. Moreover, the boundedness of $G$ on $\ell^2(\bbZ)$ implies that the symbol $\varphi$ is bounded on $\bbT$. Thus, we finally obtain
\begin{equation}
0\leq \varphi\in L^\infty(\bbT).
\label{eq:aa9}
\end{equation}
The normalisation $\norm{\psi_n}=1$ for all $n$ corresponds to the condition $\widehat{\varphi}_0=1$.
We will denote
\[
\varphi_{\min}=\essinf\varphi, \quad \varphi_{\max}=\esssup\varphi.
\]

From the foregoing discussion it is clear that $G$ uniquely determines $\varphi$ and vice versa. In what follows, \emph{we will regard $\varphi$ as the main functional parameter}. Thus, our random family of operators $\calP_\varkappa$ in \eqref{eq:aa1} is parameterised by the distribution $\bbP_0$ of \eqref{eq:P0} and the symbol $\varphi$ of \eqref{eq:sym}.

Finally, as we will explain in Section~\ref{sec:bbb}, for any symbol $\varphi$ satisfying \eqref{eq:aa9} there is a set of vectors $\{\psi_n\}$ in $\ell^2(\bbZ)$ with the Gram matrix $G=\{\widehat{\varphi}_{n-m}\}_{n,m\in\bbZ}$. In this sense, the symbol $\varphi$ is a \emph{free parameter} of our model.

\subsection{The kernel of $\calP_\varkappa$}\label{sec.b5}
There are three reasons why $\calP_\varkappa$ may have a non-trivial kernel:
\begin{enumerate}[\rm (i)]
\item
$\{\psi_n\}_{n\in\bbZ}$ are not complete in the Hilbert space;
\item
$\varkappa_n$ take value $0$ (this is ruled out by our Assumption~\ref{ass:1}(a));
\item
$G$ has a non-trivial kernel.
\end{enumerate}
If we assume that $\psi_n$ are complete (or restrict $\calP_\varkappa$ to their closed linear span), the only possible reason is (iii). We would like to point out that this reason cannot be ruled out. Indeed, one can have a symbol $\varphi$ that takes value $0$ on a set of positive measure. Then $G$ has an infinite-dimensional kernel and in this case  $\calP_\varkappa$ will have an infinite-dimensional kernel for \emph{all} $\varkappa$. 

\subsection{Location of the spectrum of $\calP_\varkappa$}
For $\Delta\subset\bbR$, we denote by $\conv(\Delta)$ the closed convex hull of $\Delta$ (i.e. the minimal closed interval containing $\Delta$). We also use notation \eqref{eq:sumset} for product sets.

\begin{theorem}\label{thm:aa2}
\begin{enumerate}[\rm (i)]
\item
The deterministic spectrum $\sigma(\calP_\varkappa)$ of $\calP_\varkappa$ satisfies
 \begin{equation}
 \supp \bbP_0\cdot\essran\varphi\subset\sigma(\calP_\varkappa)\subset \supp\bbP_0\cdot\conv(\essran\varphi). 
\label{eq:aa10}
\end{equation}
\item
The spectrum $\sigma(\calP_\varkappa)$ is a subset of the interval
$[\sigma_{\min},\sigma_{\max}]$, where
\[
\sigma_{\min}=\varkappa_{\min}\varphi_{\min}, \quad
\sigma_{\max}=\varkappa_{\max}\varphi_{\max}, 
\]
and 
\[
\sigma_{\min}=\min\sigma(\calP_\varkappa), 
\quad 
\sigma_{\max}=\max\sigma(\calP_\varkappa).
\]
In particular, if $\essran\varphi$ is an interval, we have 
\[
\sigma(\calP_\varkappa)=\supp \bbP_0\cdot\essran\varphi.
\]
\end{enumerate}
\end{theorem}
If $\varphi$ is continuous, then $\essran\varphi$ is an interval, and so the last statement of the theorem applies. 
\begin{remark}
There is a subtlety in the second inclusion in \eqref{eq:aa10} that we feel should be addressed here (for the details see Section~\ref{sec.cc}). If $A$ and $B$ are bounded self-adjoint operators, it is not true in general that $\sigma(A+B)\subset\sigma(A)+\sigma(B)$, but it is true that 
\[
\sigma(A+B)\subset\conv(\sigma(A))+\sigma(B)
\quad\text{and}\quad 
\sigma(A+B)\subset\conv(\sigma(B))+\sigma(A).
\] 
Similarly, if $A$ and $B$ are bounded \emph{positive definite} operators, one can consider their symmetrised product $\sqrt{A}B\sqrt{A}$. It is not true that $\sigma(\sqrt{A}B\sqrt{A})\subset\sigma(A)\cdot\sigma(B)$, but it is true that 
\[
\sigma(\sqrt{A}B\sqrt{A})\subset\conv(\sigma(A))\cdot\sigma(B)
\quad\text{and}\quad 
\sigma(\sqrt{A}B\sqrt{A})\subset\conv(\sigma(B))\cdot\sigma(A).
\] 
Of course, this subtlety is invisible if the spectrum of one of the two operators is an interval, which is precisely the situation in most mathematical physics applications.
\end{remark}

The proof of Theorem~\ref{thm:aa2} is given in Section~\ref{sec.cc}. Once we pass to the realisation of $\calP_\varkappa$ discussed in Section~\ref{sec:bbb}, the proof becomes fairly standard (for the mathematical physics literature) and uses some ideas from \cite{Ki-Ma:82} and \cite{Ku-So:80}.

\subsection{The integrated density of states measure}
For every $\varkappa$ and every $N$, let us denote the non-zero eigenvalues of the finite-rank approximation $\calP_\varkappa^{(N)}$ by $\lambda_k(\calP_\varkappa^{(N)})$, counting multiplicity, and consider the measure
\[
\dd\nu^{(N)}_{\calP_\varkappa}(\lambda)=\frac1{2N+1}\biggl(\sum_k \delta\bigl(\lambda-\lambda_k(\calP_\varkappa^{(N)})\bigr)\biggr)\dd\lambda,
\]
(here $\delta(\bigcdot)$ is the delta-function) 
having the point mass of $1/(2N+1)$ at each of the eigenvalues. 

We prove that the measures
$\nu^{(N)}_{\calP_\varkappa}$ converge weakly to a deterministic ($\varkappa$-independent) probability measure $\nu_{\calP_\varkappa}$ almost surely. To give the cleanest statement, below we consider the operator $\calP_\varkappa^{(N)}$ restricted onto its $(2N+1)$-dimensional range, which coincides with the linear span of $\{\psi_n\}_{n=-N}^N$. 

\begin{theorem}\label{thm:aa3}
There exists a non-random probability measure $\nu_{\calP_\varkappa}$ with $\supp \nu_{\calP_\varkappa}$ equal to the deterministic spectrum of $\calP_\varkappa$, and a set $\Omega_*\subset\Omega$ of full measure such that for any $\varkappa\in\Omega_*$ and any continuous function $f$ we have 
\begin{align}
&\lim_{N\to\infty}\frac1{2N+1}\Tr f(\calP_\varkappa^{(N)}|_{\Ran \calP_\varkappa^{(N)}}) \notag
\\&\hspace{2cm}=\lim_{N\to\infty} \int_0^\infty f(\lambda) \dd \nu_{\calP_\varkappa}^{(N)} (\lambda) 
=\int_{0}^\infty f(\lambda)\dd\nu_{\calP_\varkappa}(\lambda). 
\label{eq:cc2}
\end{align}
The first and second moments of $\nu_{\calP_\varkappa}$ are
\begin{align}
\int_0^\infty \lambda\dd\nu_{\calP_\varkappa}(\lambda)
&=
\bbE\{\varkappa_0\}, 
\label{eq:cc2a}
\\
\int_0^\infty \lambda^2\dd\nu_{\calP_\varkappa}(\lambda)
&=\bbV\{\varkappa_0\}+(\bbE\{\varkappa_0\})^2\sum_{m=-\infty}^\infty \abs{\widehat{\varphi}_m}^2,
\label{eq:cc3}
\end{align}
where $\bbV\{\varkappa_0\}=\bbE\{\varkappa_0^2\}-(\bbE\{\varkappa_0\})^2$ is the variance of $\varkappa_0$ and $\widehat{\varphi}_m$ are the Fourier coefficients of the symbol $\varphi$, defined in \eqref{eq:fphi}. Recall that $\abs{\widehat{\varphi}_m}^2=G_m$, see \eqref{eq:fphi}. 
\end{theorem}

Of course, the restriction onto the range of $\calP_\varkappa^{(N)}$ in \eqref{eq:cc2} is important only if $f(0)\not=0$. If $f(0)=0$, it can be omitted. The measure $\nu_{\calP_\varkappa}$ is called the \emph{integrated density of states (IDS) measure} for $\calP_\varkappa$. 
As it is standard in weak convergence arguments, the limiting relation \eqref{eq:cc2} can be extended to $f=\chi_{(\lambda,\infty)}$ as long as $\lambda$ is not a point mass of the measure $\nu_{\calP_\varkappa}$.

\begin{example}\label{exa:bb4}
Suppose $\psi_n$ are orthogonal; then $\nu_{\calP_\varkappa}=\bbP_0$. 
\end{example}

\begin{example}\label{exa:bb5}
Suppose $\varkappa_n$ are non-random, with $\varkappa_n=1$ for all $n$. 
Then, as discussed in Section~\ref{sec.bb}, the operator $\calP^{(N)}_1$ is unitarily equivalent to a finite truncation of the Gram matrix $G$. In this case, by the First Szeg\H{o} limit theorem, the measure $\nu_{\calP_1}$ is the push-forward of the Lebesgue measure on the unit circle by the symbol $\varphi$, i.e. 
\[
\nu_{\calP_1}(\Delta)=\frac1{2\pi}\int_{\{k: \varphi(k)\in\Delta\}}\dd k. 
\]
\end{example}

The proof of Theorem~\ref{thm:aa3} is given in Section~\ref{sec.cc}. Similarly to Theorem~\ref{thm:aa2}, once we pass to the realisation of $\calP_\varkappa$ discussed in Section~\ref{sec:bbb}, the proof becomes a matter of appealing to the standard general results on spectral theory of ergodic operators. As a result, Theorem~\ref{thm:aa3} is valid for any bounded ergodic sequence $\{\varkappa_n\}_{n\in\bbZ}$, not necessarily for a sequence of i.i.d. variables. 

\subsection{A Wegner type estimate for $\nu_{\calP_\varkappa}$}
Example~\ref{exa:bb4} and a little reflection suggest that if $\bbP_0$ in \eqref{eq:P0} is absolutely continuous, i.e. 
\[
\dd\bbP_0(\lambda)=\rho(\lambda)\dd\lambda
\]
with some density function $\rho$, then we should expect $\nu_{\calP_\varkappa}$ to be absolutely continuous. This is indeed the case; in mathematical physics literature, this is called the Wegner estimate, \cite[Theorem~4.21]{Pa-Fi:92}. 
\begin{theorem}\label{thm:aa4}
Assume that $\bbP_0$ is absolutely continuous with the density $\rho$ uniformly bounded by $\rho_{\max}<\infty$. Then the IDS measure $\nu_{\calP_\varkappa}$ is absolutely continuous with the density satisfying the Wegner type estimate
\[
\frac{\dd\nu_{\calP_\varkappa}(\lambda)}{\dd\lambda}
\leq 
\frac{\rho_{\max}\varkappa_{\max}}{\lambda}, \quad \lambda>0.
\]
\end{theorem}
The proof is given in Section~\ref{sec:e}. 
We note that this Wegner type estimate is independent of the symbol $\varphi$. 

\subsection{Point masses of $\nu_{\calP_\varkappa}$}
Continuing this line of thought, it is reasonable to ask whether $\nu_{\calP_\varkappa}$ is continuous if the distribution $\bbP_0$ is less regular than the one described in Theorem~\ref{thm:aa4}, while $\varphi$ is non-constant. For comparison, we note that for Anderson model, the IDS measure is known to be logarithmically H\"older continuous \cite[Theorem~11.11]{Pa-Fi:92}, but the proof relies on the Thouless formula -- the mechanism that seems to be absent in our model (see also \cite[Theorem~3.4]{Pa-Fi:92} for further results in this direction). It may be surprising therefore that in our model, point masses of $\nu_{\calP_\varkappa}$ can occur. As above, we denote $\sigma_{\max}=\varkappa_{\max}\varphi_{\max}$. 

\begin{theorem}\label{thm:aa4a}
\begin{enumerate}[\rm (i)]
\item
Suppose $\varphi(t)=\varphi_{\max}$ on an arc of length $2a$, $0<a<\pi$, of the unit circle, and 
\begin{equation}
\bbP_0(\{\varkappa_{\max}\})+\frac{a}\pi>1.
\label{eq:aa17}
\end{equation}
Then $\nu_{\calP_\varkappa}(\{\sigma_{\max}\})>0$.
\item 
Suppose that the set $\{t\in\bbT: \varphi(t)=\varphi_{\max}\}$ is contained in an arc of length $2a$, $0<a<\pi$ (up to sets of measure zero on $\bbT$) and 
\[
\bbP_0(\{\varkappa_{\max}\})+\frac{a}\pi\leq 1.
\]
Then $\nu_{\calP_\varkappa}(\{\sigma_{\max}\})=0$. 
In particular, this is the case if the set $\{t\in\bbT: \varphi(t)=\varphi_{\max}\}$ has measure zero. 
\end{enumerate}
\end{theorem}

One can similarly treat a point mass of $\nu_{\calP_\varkappa}$ at the bottom of the spectrum $\sigma_{\min}$. However, our construction doesn't work for interior points of the spectrum. 

The proof is given in Section~\ref{sec.dd}. It relies on the deep classical theorem of Beurling and Malliavin on the completeness of a system of exponentials.

\subsection{Lifshitz tails asymptotics of the IDS measure $\nu_{\calP_\varkappa}$}
Let us consider the situation discussed in part (ii) of the previous theorem, assuming that the set $\{k\in\bbT: \varphi(k)=\varphi_{\max}\}$ has measure zero. We already know that $\nu_{\calP_\varkappa}(\{\sigma_{\max}\})=0$; now consider the behaviour of $\nu_{\calP_\varkappa}$ near $\sigma_{\max}$. According to the definition of $\nu_{\calP_\varkappa}$, we should evaluate the number of eigenvalues of $\calP_\varkappa^{(N)}$ near $\sigma_{\max}$ as $N\to\infty$. This forces all variables $\varkappa_n$ for $\abs{n}\leq N$ to take values near $\varkappa_{\max}$; this is a rare event. This line of thought led the theoretical physicist I.M.Lifshitz to conjecture that the IDS  of the Schr\"odinger operator with random potential is exponentially thin near the bottom of the spectrum. This became known as Lifshitz tails. The logic is applicable to either the top or the bottom of the spectrum. 

We prove the exponential asymptotics of $\nu_{\calP_\varkappa}$ (the Lifshitz tails asymptotics) near the top of the spectrum $\sigma_{\max}$ under some regularity assumptions on the symbol $\varphi$.

\begin{theorem}\label{thm:aa5}
Assume that the symbol $\varphi$ is $C^2$-smooth on $\bbT$ and the maximum of $\varphi$ is non-degenerate and is attained at a single point $k_{\max}$. Further, assume that $\bbP_0$ is not supported at one point, i.e. $\varkappa_{\min}<\varkappa_{\max}$, and 
\begin{equation}
\bbP_0((\varkappa_{\max}-\eps,\varkappa_{\max}])\geq C\eps^\ell
\label{eq:aa11a}
\end{equation}
for some $\ell>0$. Then
\begin{align}
\lim_{\delta\to0_+} \log \bigl(-\log \nu_{\calP_\varkappa}((\sigma_{\max}-\delta,\sigma_{\max}])\bigr)/\log\delta = -\frac12.
\label{eq:aa12} 
\end{align}
\end{theorem}
Roughly speaking, \eqref{eq:aa12} says that 
\[
\nu_{\calP_\varkappa}((\sigma_{\max}-\delta,\sigma_{\max}])\sim \ee^{-C\delta^{-1/2}}
\]
as $\delta\to0$. 
In the same way, one can prove Lifshitz tails asymptotics at the lower edge of the spectrum:
\[
\nu_{\calP_\varkappa}([\sigma_{\min},\sigma_{\min}+\delta))\sim \ee^{-C\delta^{-1/2}}
\]
as $\delta\to0$, under similar assumptions on $\varphi$ near $\varphi_{\min}$ and on $\bbP_0$ near $\varkappa_{\min}$.

Assumptions on $\varphi$ can be relaxed, allowing multiple points of maximum. However, if the maximum is degenerate, the exponent $1/2$ in the asymptotics is expected to change; see  \cite{GRM1} and \cite[Section 10.B and 10.C]{Pa-Fi:92} for a similar result in this direction. 

Roughly speaking, \eqref{eq:aa11a} means that $\bbP_0$ is ``not too thin'' near $\varkappa_{\max}$. For example, if $\bbP_0$ has a point mass at $\varkappa_{\max}$, then \eqref{eq:aa11a} is satisfied. 

The proof of Theorem~\ref{thm:aa5} is given in Section~\ref{sec.d}. It consists of two steps: estimating $\varphi$ by ``model symbols'' of the form $a+b\cos k$ (this idea is borrowed from \cite{GRM1}; see also \cite{Klopp:98}) and using the Dirichlet--Neumann bracketing approach as presented in \cite{Simon} for the model symbols.

\subsection{Anderson type localisation}
It is obvious that if $\{\psi_n\}_{n\in\bbZ}$ are orthogonal to one another (i.e. if the Gram matrix $G$ is diagonal), then for all $\varkappa$ the spectrum of the operator $\calP_\varkappa$ is pure point and $\{\psi_n\}_{n\in\bbZ}$ are the corresponding eigenvectors. Here we prove that if $\{\psi_n\}_{n\in\bbZ}$ are \emph{nearly-orthogonal}, and the distribution of $\varkappa$ is sufficiently regular, then the spectrum of $\calP_\varkappa$ is still pure point. 

\begin{theorem}\label{thm:aa6}
Assume that $\bbP_0$ of \eqref{eq:P0} is uniformly H\"older continuous on $[\varkappa_{\min},\varkappa_{\max}]$:
\begin{equation}
\bbP_0([x-\delta,x+\delta])\leq C\delta^\tau, \quad \forall x\in [\varkappa_{\min},\varkappa_{\max}]
\label{eq:g}
\end{equation}
with some $\tau\in(0,1]$ and $C>0$ independent of $x$. Then there exists a sufficiently small $\eta>0$ (depending on $\bbP_0$) such that for any symbol $\varphi$ satisfying
\begin{equation}
\sum_{n\not=0}\abs{\widehat{\varphi}_n}^{\tau/4}<\eta,
\label{eq:g1}
\end{equation}
the spectrum of $\calP_\varkappa$ is pure point. 
\end{theorem}
We recall that by our assumption that $\psi_n$ are normalised, we have $\widehat{\varphi}_0=1$, while the sum in \eqref{eq:g1} is over $n\not=0$. The exponent $\tau/4$ in \eqref{eq:g1} can be replaced by any positive number less than $\tau/2$.

This theorem should be contrasted with the opposite scenario: suppose $\varkappa$ is deterministic, e.g. $\varkappa_n=1$ for all $n$ and $\varphi$ is smooth and non-constant on any open interval. Then the spectrum of $\calP_\varkappa$ is purely absolutely continuous.

We follow the fractional moments method of Aizenman--Molchanov (see the original paper \cite{Aiz-Mol} or the monograph \cite{Aiz-War}), developed for the Anderson model. Our argument and results have some common points with \cite{Fi-Pa:90,HKK}.

\section{Realisations of $\calP_\varkappa$}
\label{sec:bbb}

As already mentioned, the choice of the Hilbert space $\calH$ and of the sequence $\{\psi_n\}$ are immaterial to our model, as long as the Gram matrix is fixed. Below we discuss some concrete choices, i.e. \emph{realisations} of our model.

\subsection{Realisation in $\ell^2(\bbZ)$}
By \eqref{eq:aa9}, we can define a bounded non-negative function $\sqrt{\varphi}$ on $\bbT$. 
For typographical reasons, we will write 
\begin{equation}
\theta=\sqrt{\varphi}.
\label{eq:tp}
\end{equation}
Let us take $\calH=\ell^2(\bbZ)$ and 
\[
\psi_n=L_{\theta}\delta_n, \quad n\in\bbZ,
\]
where $L_\theta$ is the Laurent operator $\{\widehat{\theta}_{n-m}\}_{n,m\in\bbZ}$ and $\{\delta_n\}$ is the canonical basis in $\ell^2(\bbZ)$. To check that this is indeed a realisation of our model, we need to compute the Gram matrix of the sequence $\{\psi_n\}_{n\in\bbZ}$ introduced above. We notice first that 
\[
(L_{\theta})^2=L_{\varphi}.
\]
Using this, we have 
\[
\jap{L_{\theta}\delta_m,L_{\theta}\delta_n}_{\ell^2(\bbZ)}
=\jap{(L_{\theta})^2 \delta_m,\delta_n}_{\ell^2(\bbZ)}
=\jap{L_\varphi \delta_m,\delta_n}_{\ell^2(\bbZ)}
=\widehat{\varphi}_{n-m},
\]
as needed, see \eqref{eq:aa31}. Thus, we can write the realisation of $\calP_\varkappa$ in $\ell^2(\bbZ)$ as follows:
\[
\calP_\varkappa
=\sum_{n=-\infty}^\infty \varkappa_n \jap{\bigcdot, L_{\theta}\delta_n}L_{\theta}\delta_n
=L_{\theta}\left(  \sum_{n=-\infty}^\infty \varkappa_n \jap{\bigcdot, \delta_n}\delta_n
\right)L_{\theta}.
\]
Of course, the operator in parentheses here is simply the operator of multiplication by the sequence $\{\varkappa_n\}_{n\in\bbZ}$; by the standard abuse of notation, we will denote this operator simply by $\varkappa$. Then we can write the above formula as
\begin{equation}
\calP_\varkappa =L_{\theta}\varkappa L_{\theta}
\quad \text{ in $\ell^2(\bbZ)$.}
\label{eq:bb1}
\end{equation}

\subsection{A unitarily equivalent form of realisation in $\ell^2(\bbZ)$}
\label{sec:bbb2}
We recall that for any bounded operator $X$ in a Hilbert space, the restrictions 
\[
X^*X|_{(\Ker X)^{\perp}}
\quad \text{ and }\quad
XX^*|_{(\Ker X^*)^{\perp}}
\]
are unitarily equivalent. (We have already seen an example of the use of this principle in \eqref{eq:aa5}.) As a shorthand, we will say that $X^*X$ and $XX^*$ are unitarily equivalent \emph{modulo kernels}. If in addition $\dim\Ker X=\dim\Ker X^*$, then $X^*X$ and $XX^*$ are unitarily equivalent. 

We recall that by our Assumption~\ref{ass:1}(a), all variables $\varkappa_n$ are positive. Thus, we can write \eqref{eq:bb1} as
\[
\calP_\varkappa = (\sqrt{\varkappa} L_{\theta})^*(\sqrt{\varkappa} L_{\theta})
\]
and therefore $\calP_\varkappa$ is unitarily equivalent, modulo kernels, to the operator 
\begin{equation}
\myQ_\varkappa
:=(\sqrt{\varkappa} L_{\theta})(\sqrt{\varkappa} L_{\theta})^*
=\sqrt{\varkappa}(L_{\theta})^2\sqrt{\varkappa}
=\sqrt{\varkappa}L_{\varphi}\sqrt{\varkappa}
\label{eq:bb2}
\end{equation}
in $\ell^2(\bbZ)$. Of course, ${\myQ}_\varkappa$ is no longer a sum of rank-one projections, but it is nevertheless a useful ``model'' of $\calP_\varkappa$. Explicitly, ${\myQ}_\varkappa$ is the doubly-infinite matrix 
\[
{\myQ}_\varkappa
=
\{\sqrt{\varkappa_n}\widehat{\varphi}_{n-m}\sqrt{\varkappa_m}\}_{n,m\in\bbZ},
\]
which can be viewed as a random multiplicative perturbation of the Laurent operator $L_\varphi$. 
We shall soon see that ${\myQ}_\varkappa$ is also an ergodic operator.

\begin{remark}\label{rmk:bb1}
Inspecting the above definitions and using the fact that $\varkappa_n\geq\varkappa_{\min}>0$ for all $n$, we find
\[
\Ker\calP_\varkappa=\Ker L_\varphi,\quad
\Ker{\myQ}_\varkappa=\varkappa^{-1/2}\Ker L_\varphi,
\]
and so the dimensions of the kernels of $\calP_\varkappa$ and ${\myQ}_\varkappa$ coincide. 
Thus, $\calP_\varkappa$ and ${\myQ}_\varkappa$ are unitarily equivalent for \emph{every} $\varkappa$. 
\end{remark}

\subsection{Realisation in $L^2(\bbT)$}
By using the Fourier series correspondence between $L^2(\bbT)$ and $\ell^2(\bbZ)$, we can map the realisation \eqref{eq:bb1} into $L^2(\bbT)$. Indeed, let $\{e_n\}_{n\in\bbZ}$ be the standard trigonometric basis, 
\[
e_n(k)=\ee^{\ii nk}, \quad n\in\bbZ, \quad -\pi\leq k\leq \pi.
\]
We normalise the Lebesgue measure on $\bbT$ to one, so that with this normalisation, $e_n$ are orthonormal. We then set $\calH=L^2(\bbT)$ and (see \eqref{eq:tp})
\begin{equation*}
\psi_n(k)=\theta(k)e_n(k). 
\end{equation*}
It is evident that 
\[
\jap{\theta e_m,\theta e_n}_{L^2(\bbT)}=\jap{\varphi e_m,e_n}_{L^2(\bbT)}=\widehat{\varphi}_{n-m},
\]
and so this is indeed a realisation of our model $\calP_\varkappa$ of \eqref{eq:aa1}. 
We can write
\[
\calP_\varkappa=\theta\left(\sum_{n=-\infty}^\infty \varkappa_n\jap{\bigcdot,e_n}e_n\right)\theta
\quad \text{ in $L^2(\bbT)$.}
\]
We will not use this realisation in this paper, but it may be useful elsewhere. 

\section{Location of the spectrum and the IDS measure: proofs of Theorems~\ref{thm:aa2} and \ref{thm:aa3}}
\label{sec.cc}

\subsection{Spectra of multiplicative perturbations}
In preparation for the proof of Theorem~\ref{thm:aa2}, here we give a general result on the spectra of multiplicative perturbations of self-adjoint operators. We use the notation \eqref{eq:sumset} for sum and product sets. First we need
\begin{proposition}\cite[Chapter V, Theorem 4.10]{Kato}
\label{prp:cc1}
Let $A$ and $B$ be self-adjoint operators on a Hilbert space, with $A$ bounded, $\norm{A}\leq a$. Then $\dist(\sigma(A+B),\sigma(B))\leq a$. Equivalently, $\sigma(A+B)\subset [-a,a]+\sigma(B)$. 
\end{proposition}
By adding a multiple of the identity operator to $A$, we obtain 
\[
\sigma(A+B)\subset\conv(\sigma(A))+\sigma(B).
\]
The analogue of this fact for multiplicative perturbations is 

\begin{theorem}\label{thm:add2}
Let $A$ and $B$ be bounded positive semi-definite operators on a Hilbert space. Then 
\begin{equation}
\sigma(\sqrt{B}A\sqrt{B})\subset\conv(\sigma(A))\cdot\sigma(B).
\label{eq:add1}
\end{equation}
\end{theorem}
This statement is probably known to experts, but despite our best efforts, we could not find its proof in existing literature. 
\begin{proof}
As a first step, assume that both $A$ and $B$ have positive lower bounds, i.e. $a_- I\leq A$ and $b_- I\leq B$ with $a_-$, $b_-$ positive  (here $I$ is the identity operator). 
Write $\conv(\sigma(A))=[a_-,a_+]$ for $0<a_-\leq a_+$. Since $a_- I\leq A\leq a_+ I$,  
denoting for brevity $M=\sqrt{B}A\sqrt{B}$, we find
\begin{equation}
0\leq a_- B\leq M\leq a_+ B.
\label{eq:add3}
\end{equation}
It is known that the function $x\mapsto \log x$ is operator-positive on $(0,\infty)$ (see e.g. \cite[Theorem~4.1]{Simon-Lowner}), i.e. if operators $X$ and $Y$ satisfy $0\leq X\leq Y$, then $\log X\leq \log Y$. Applying this to \eqref{eq:add3}, we find 
\[
\log a_-+\log B\leq \log M\leq \log a_++\log B,
\]
where both operators $\log M$ and $\log B$ are bounded by our assumptions. 
Subtracting the average of left- and right-hand sides, we find 
\[
\norm{\log M-\log B-\log \sqrt{a_-a_+}}\leq\Delta, \quad\text{ where }\Delta=\log \sqrt{a_+/a_-}. 
\]
Applying Proposition~\ref{prp:cc1}, we find that 
\[
\sigma(\log M)\subset [-\Delta,\Delta]+\sigma\bigl(\log (\sqrt{a_-a_+}B)\bigr).
\]
Exponentiating and using the spectral mapping theorem for self-adjoint operators, we find
\[
\sigma(M)
\subset [\sqrt{a_-/a_+},\sqrt{a_+/a_-}]\cdot\sigma(\sqrt{a_-a_+}B)
=[a_-,a_+]\cdot\sigma(B),
\]
as required. 

Now let us lift the assumption that $A$ and $B$ have positive lower bounds. For the operators $A_\eps=A+\eps I$ and $B_\eps=B+\eps I$ with $\eps>0$ we have 
\begin{equation}
\sigma(\sqrt{B_\eps}A_\eps\sqrt{B_\eps})\subset\conv(\sigma(A_\eps))\cdot\sigma(B_\eps)
\label{eq:add2}
\end{equation}
by the already proved step. Since
\[
\norm{\sqrt{B_\eps}A_\eps\sqrt{B_\eps}-\sqrt{B}A\sqrt{B}}\to0,
\quad \eps\to0,
\]
applying Proposition~\ref{prp:cc1} to \eqref{eq:add2}, we obtain \eqref{eq:add1}. 
\end{proof}

\subsection{Proof of Theorem~\ref{thm:aa2}}
We need to prove two inclusions in \eqref{eq:aa10}.

\smallskip

\emph{Proof of the first inclusion.}
Take some 
\[
\lambda_1\in\essran\varphi=\sigma(L_\varphi)
\quad\text{ and }\quad
\lambda_2\in\supp\bbP_0,
\]
let $\lambda=\lambda_1\lambda_2$ and let us prove that $\lambda\in\sigma(\calP_\varkappa)$. 
We use the realisation \eqref{eq:bb1}
\[
\calP_\varkappa =L_{\theta}\varkappa L_{\theta}
\] 
in $\ell^2(\bbZ)$. 
It suffices to construct a sequence $\{w_n\}_{n=1}^\infty$ of elements in $\ell^2(\bbZ)$ such that $\inf_n\norm{w_n}>0$ and 
\begin{equation}
\norm{\calP_{\varkappa}w_n-\lambda w_n}\to0, \quad n\to\infty
\label{eq:cc11}
\end{equation}
a.s. on $\Omega$. 
Since $L_\varphi=(L_{\theta})^2$, we have $\sqrt{\lambda_1}\in\sigma(L_{\theta})$. Let us select a sequence $v_n\in\ell^2(\bbZ)$ such that $0<c\leq\norm{v_n}\leq C$ and 
\[
\norm{L_{\theta} v_n-\sqrt{\lambda_1}v_n}\to0, \quad n\to\infty.
\]
Truncating $v_n$ if necessary, we can assume that each $v_n$ has compact support. We also note that, by the shift invariance of $L_{\theta}$, any shifts of $v_n$ satisfy the same property. 

Next, by a standard argument (see e.g. \cite[Theorem~3.12]{Aiz-War}) one can find intervals in $\bbZ$ of arbitrary length where $\varkappa$ is arbitrarily close to $\lambda_2$. More precisely: there is a set of $\varkappa\in\Omega$ of full measure such that for all $\varkappa$ in this set, for all $n\geq1$ there exists $j_n\in\bbZ$ such that 
\[
\sup_{i\in\supp v_n}\abs{\varkappa_{i+j_n}-\lambda_2}<1/n.
\]
Now denote 
\[
w_n(i)=v_n(i-j_n); 
\]
then $\norm{w_n}=\norm{v_n}$ and $w_n$ satisfy 
\[
\norm{L_{\theta} w_n-\sqrt{\lambda_1}w_n}\to0, \quad
\norm{\varkappa w_n-\lambda_2 w_n}\to0, \quad n\to\infty.
\]
From here we immediately obtain \eqref{eq:cc11}. 

\smallskip

\emph{Proof of the second inclusion.}
For a.e. $\varkappa$, we have $\sigma(\varkappa)=\supp\bbP_0$. We also have $\sigma(L_\varphi)=\essran\varphi$. Now recalling that $\calP_\varkappa=\sqrt{L_\varphi}\varkappa\sqrt{L_\varphi}$, the required result follows by Theorem~\ref{thm:add2}. The proof of Theorem~\ref{thm:aa2} is complete.
 \qed

\subsection{Proof of Theorem~\ref{thm:aa3}: the existence of $\nu_{\calP_\varkappa}$ and formula \eqref{eq:cc2}}
\label{sec:cc1}
Consider the family of operators ${\myQ}_\varkappa$ in $\ell^2(\bbZ)$, defined by \eqref{eq:bb2}. Let $U$ be the shift operator in $\ell^2(\bbZ)$, acting on elements of the canonical basis as 
\[
U\delta_n=\delta_{n-1}. 
\]
Since the Laurent operator $L_\varphi$ commutes with $U$, we find 
\[
{\myQ}_{T\varkappa}=U{\myQ}_{\varkappa}U^*
\]
and so ${\myQ}_{\varkappa}$ is an ergodic operator in $\ell^2(\bbZ)$, see Appendix~\ref{sec:A2}. Following the general definition outlined there (see \eqref{eq:A3}), we define the IDS measure $\nu_{{\myQ}_\varkappa}$ of ${\myQ}_\varkappa$ by 
\begin{equation}
\nu_{{\myQ}_\varkappa}(\Delta)=\bbE\{\jap{\chi_\Delta({\myQ}_\varkappa)\delta_0,\delta_0}\},
\label{eq:cc0}
\end{equation}
where $\Delta\subset\bbR$ is any Borel set. The support of $\nu_{{\myQ}_\varkappa}$ coincides with the deterministic spectrum of ${\myQ}_\varkappa$. Since ${\calP}_\varkappa$ and ${\myQ}_\varkappa$ are unitarily equivalent (see Remark~\ref{rmk:bb1}), it follows that the support of $\nu_{{\myQ}_\varkappa}$ coincides with the deterministic spectrum of $\calP_\varkappa$. 

It remains to prove that $\nu_{\calP_\varkappa}$ in the right-hand side of \eqref{eq:cc2} coincides with $\nu_{{\myQ}_\varkappa}$. Applying Proposition~\ref{prp:A2}, we get
\begin{equation}
\int_{0}^\infty f(\lambda)\dd\nu_{{\myQ}_\varkappa}(\lambda)=
\lim_{N\to\infty}
\frac1{2N+1}\Tr f({\myQ}_\varkappa^{(N)})
\label{eq:cc13}
\end{equation}
for any continuous function $f$, where ${\myQ}_\varkappa^{(N)}$ is the $(2N+1) \times (2N+1)$ matrix
\[
{\myQ}_\varkappa^{(N)}
=\{\jap{{\myQ}_\varkappa\delta_n,\delta_m}\}_{n,m=-N}^N
=\{\sqrt{\varkappa_n}\widehat{\varphi}_{n-m}\sqrt{\varkappa_m}\}_{n,m=-N}^N
\]
We need to replace ${\myQ}_\varkappa^{(N)}$ by $\calP_\varkappa^{(N)}|_{\Ran \calP_\varkappa^{(N)}}$ in the right-hand side of \eqref{eq:cc13}. Let us check that 
\begin{equation}
\Tr f(\calP_\varkappa^{(N)}|_{\Ran \calP_\varkappa^{(N)}})=\Tr f({\myQ}_\varkappa^{(N)})
\label{eq:cc12}
\end{equation}
for any continuous function $f$. If $f=\mathrm{const}$, then \eqref{eq:cc12} clearly holds. Thus, it suffices to prove this relation for any continuous $f$ vanishing at zero, in which case the restriction onto the range of $\calP_\varkappa^{(N)}$ is not necessary. We denote by $1_N$ the orthogonal projection in $\ell^2(\bbZ)$ onto the subspace spanned by $\delta_n$, $-N\leq n\leq N$. Modifying \eqref{eq:bb1}, we can write
\[
\calP_\varkappa^{(N)}=L_{\theta}\varkappa 1_N L_{\theta}
\]
(recall that $(L_{\theta})^2=L_\varphi$).
For any $m\geq1$ we have 
\begin{align*}
(\calP_\varkappa^{(N)})^m
&=(L_{\theta}\sqrt{\varkappa}1_N\sqrt{\varkappa}L_{\theta})^m
=L_{\theta}\sqrt{\varkappa}(1_N\sqrt{\varkappa}L_\varphi\sqrt{\varkappa}1_N)^{m-1}\sqrt{\varkappa}L_{\theta}
\\
&=L_{\theta}\sqrt{\varkappa}(1_N{\myQ}_\varkappa 1_N)^{m-1}\sqrt{\varkappa}L_{\theta}
\end{align*}
and therefore, by the cyclicity of trace, 
\[
\Tr
(\calP_\varkappa^{(N)})^m
=
\Tr L_{\theta}\sqrt{\varkappa}(1_N{\myQ}_\varkappa 1_N)^{m-1}\sqrt{\varkappa}L_{\theta}
=
\Tr(1_N{\myQ}_\varkappa 1_N)^{m}
=
\Tr({\myQ}_\varkappa^{(N)})^{m}.
\]
From here we obtain \eqref{eq:cc12} for any polynomial $f$ with $f(0)=0$. The proof is finished by the application of the Weierstrass approximation theorem. 

In view of the above, we will henceforth  write
\[
\nu_{{\myQ}_\varkappa}=\nu_{\calP_\varkappa}=\nu.
\]

\subsection{The moments of $\nu$: proof of \eqref{eq:cc2a} and \eqref{eq:cc3}}\label{sec:cc2}
For the first moment of $\nu$, we have by the law of large numbers for $\varkappa$ 
\[
\int_0^\infty \lambda\dd\nu(\lambda)=
\lim_{N\to\infty}\frac1{2N+1}\Tr \calP_\varkappa^{(N)}
=\lim_{N\to\infty}\frac1{2N+1}\sum_{n=-N}^N \varkappa_n
=\bbE\{\varkappa_0\}.
\]
For the second moment we have, in view of \eqref{eq:ppg}
\begin{align*}
\int_0^\infty \lambda^2\dd\nu(\lambda)
=\bbE\{\jap{({\myQ}_\varkappa)^2\delta_0,\delta_0}\}
=\bbE\left\{\sum_{m\in\bbZ}\varkappa_0\varkappa_m\abs{\widehat{\varphi}_m}^2\right\}.
\end{align*}
Separating the term with $m=0$, recalling that $\widehat{\varphi}_0=1$ and using the independence of $\varkappa_0$ and $\varkappa_m$ with $m\not=0$, we find  
\begin{align*}
\bbE\left\{\sum_{m\in\bbZ}\varkappa_0\varkappa_m\abs{\widehat{\varphi}_m}^2\right\}
&=
\bbE\{\varkappa_0^2\}
+
\bbE\left\{\sum_{m\not=0}\varkappa_0\varkappa_m\abs{\widehat{\varphi}_m}^2\right\}
\\
&=
\bbE\{\varkappa_0^2\}
+
(\bbE\{\varkappa_0\})^2
\sum_{m\not=0}\abs{\widehat{\varphi}_m}^2
\\
&=
\bbV\{\varkappa_0\}+(\bbE\{\varkappa_0\})^2
\sum_{m}\abs{\widehat{\varphi}_m}^2,
\end{align*}
as claimed. 
The proof of Theorem~\ref{thm:aa3} is complete. \qed

\section{Wegner estimate: Proof of Theorem~\ref{thm:aa4}}
\label{sec:e}

We use both realisations introduced above
\[
\calP_\varkappa =L_{\theta}\varkappa L_{\theta},
\quad
{\myQ}_\varkappa=\sqrt{\varkappa}L_\varphi\sqrt{\varkappa}, \quad  \varphi=\theta^2
\]
in $\ell^2(\bbZ)$. 
For any $\varkappa\in\Omega$, $n\in\bbZ$, and a Borel set $\Delta\subset\bbR$, let us denote 
\[
\nu_n(\Delta;\varkappa)=\jap{\chi_{\Delta}(\calP_\varkappa)\psi_n,\psi_n},
\quad
\widetilde{\nu}_n(\Delta;\varkappa)=\jap{\chi_{\Delta}({\myQ}_\varkappa)\delta_n,\delta_n};
\]
we recall that $\psi_n=L_{\theta}\delta_n$. We also recall (see \eqref{eq:cc0}) that the IDS measure $\nu_{\calP_\varkappa}=\nu$ is 
\[
\nu(\Delta)=\bbE\{\widetilde{\nu}_0(\Delta;\varkappa)\}, 
\quad 
\widetilde{\nu}_0(\Delta;\varkappa)=\jap{\chi_\Delta(\myQ_\varkappa)\delta_0,\delta_0}.
\]
Let us establish the relation between the measures $\nu_n$ and $\widetilde{\nu}_n$. 
\begin{lemma}\label{lma:e3}
For any $\varkappa\in\Omega$, $n\in\bbZ$, we have
\[
\lambda\dd\widetilde{\nu}_n(\lambda;\varkappa)
=
\varkappa_n
\dd\nu_n(\lambda;\varkappa).
\]
\end{lemma}
\begin{proof}
It suffices to prove that for all $m\geq0$, 
\[
\int_0^\infty \lambda^{m+1}\dd\widetilde{\nu}_n(\lambda;\varkappa)
=
\varkappa_n\int_0^\infty \lambda^m\dd\nu_n(\lambda;\varkappa).
\]
Similarly to the argument of Section~\ref{sec:cc1}, 
\begin{align}
{\myQ}_\varkappa^{m+1}
&=
(\sqrt{\varkappa}L_\varphi\sqrt{\varkappa})^{m+1}
=
\sqrt{\varkappa}L_\theta(L_\theta\varkappa L_\theta)^mL_\theta\sqrt{\varkappa}
\notag
\\
&=
\sqrt{\varkappa}L_{\theta}\calP_\varkappa^m L_{\theta}\sqrt{\varkappa}
\label{eq:e9}
\end{align}
and therefore
\begin{align*}
\int_0^\infty \lambda^{m+1}\dd\widetilde{\nu}_n(\lambda;\varkappa)
&=
\jap{{\myQ}_\varkappa^{m+1}\delta_n,\delta_n}
=
\jap{ \calP_\varkappa^m L_{\theta} \sqrt{\varkappa}\delta_n,L_{\theta}\sqrt{\varkappa}\delta_n}
\\
&=
\varkappa_n\jap{\calP_\varkappa^m L_{\theta} \delta_n,L_{\theta} \delta_n}
=
\varkappa_n\jap{\calP_\varkappa^m \psi_n,\psi_n}
\\
&=
\varkappa_n\int_0^\infty \lambda^{m}\dd\nu_n(\lambda;\varkappa),
\end{align*}
as required.
\end{proof}
Next, as usual in the arguments of this genre, one of the key components of the proof is the statement known as \emph{spectral averaging}, see \cite{Si-Wo:86} and references therein for the history.

\begin{proposition}\label{prp:e2}
Let $A$ be a bounded self-adjoint operator in a Hilbert space, and let $x$ be a non-zero vector in this space. For $\alpha\in\bbR$, denote $A_\alpha=A+\alpha\jap{\bigcdot,x}x$ and  consider the measure
\[
\mu_\alpha(\Delta)=\jap{\chi_{\Delta}(A_\alpha)x,x}
\]
on the real line. Then 
\begin{equation}
\int_{-\infty}^\infty \mu_\alpha(\Delta)\dd\alpha=m(\Delta),
\label{eq:e2}
\end{equation}
where the right-hand side is the Lebesgue measure of $\Delta$. 
\end{proposition}
We note that \eqref{eq:e2} should be understood in the weak sense, i.e. 
\[
\int_{-\infty}^\infty \int_{-\infty}^\infty f(\lambda)\dd\mu_\alpha(\lambda)\, \dd\alpha=\int_{-\infty}^\infty f(\lambda)\dd \lambda
\]
for a suitable class of functions $f$.

We are now ready to prove Theorem~\ref{thm:aa4} (the Wegner estimate).
\begin{proof}[Proof of Theorem~\ref{thm:aa4}]
By our assumptions, we have 
\[
\dd\bbP_0(\lambda)=\rho(\lambda)\dd\lambda, 
\]
where $\supp \rho\subset[\varkappa_{\min},\varkappa_{\max}]$ and $\norm{\rho}_\infty=\rho_{\max}<\infty$.

We write formally the sequence $\varkappa$ as a pair $(\varkappa_0,\varkappa_{\not=0})$, where $\varkappa_{\not=0}$ denotes the sequence with $n\not=0$. Then for any interval $\Delta$, taking expectation with respect to $\varkappa_0$ first, we find
\begin{align*}
\int_{\varkappa_{\min}}^{\varkappa_{\max}}\nu_0(\Delta;\varkappa_0,\varkappa_{\not=0})\rho(\varkappa_0)\dd\varkappa_0
\leq
\rho_{\max}\int_{-\infty}^\infty \nu_0(\Delta;\varkappa_0,\varkappa_{\not=0})\dd\varkappa_0
=\rho_{\max} m(\Delta),
\end{align*}
where we have used the spectral averaging formula \eqref{eq:e2} at the last step and $m(\Delta)$ is the Lebesgue measure of $\Delta$. Taking expectation with respect to $\varkappa_{\not=0}$ as well, we obtain 
\begin{equation}
\bbE\{\nu_0(\Delta;\varkappa)\}\leq \rho_{\max} m(\Delta).
\label{eq:e7}
\end{equation}
Finally, let us take $n=0$ in Lemma~\ref{lma:e3} and integrate over $\Delta$; we obtain
\begin{equation}
\widetilde{\nu}_0(\Delta;\varkappa)=\varkappa_0\int_\Delta\frac{\dd\nu_0(\lambda;\varkappa)}{\lambda}.
\label{eq:e8}
\end{equation}
Now using \eqref{eq:e7} and \eqref{eq:e8}, we can complete:
\begin{align*}
\nu(\Delta)
&=\bbE\{\widetilde{\nu}_0(\Delta;\varkappa)\}
=\bbE\left\{ \varkappa_0\int_\Delta\frac{\dd\nu_0(\lambda;\varkappa)}{\lambda}\right\}
\\
&\leq\varkappa_{\max}\bbE\left\{ \int_\Delta\frac{\dd\nu_0(\lambda;\varkappa)}{\lambda}\right\}
\leq \varkappa_{\max}\rho_{\max}\int_\Delta\frac{\dd\lambda}{\lambda}.
\end{align*}
The proof of Theorem~\ref{thm:aa4} is complete.
\end{proof}

\section{Point masses of $\nu$: proof of Theorem~\ref{thm:aa4a}}
\label{sec.dd} 

\subsection{The theorem of Beurling and Malliavin}
The proof is based on a deep theorem due to Beurling and Malliavin \cite{Be-Ma} on completeness of exponentials. 

We need a few definitions. For a discrete set $\Lambda\subset\bbR$ (i.e. $\Lambda$ has no finite points of accumulation) we define the counting function of its positive part by 
\[
n_\Lambda(t)=\#(\Lambda\cap(0,t)), \quad t>0,
\]
where $\#$ denotes the number of elements in a set. In the same way, $n_{-\Lambda}(t)$ is the counting function of the negative part of $\Lambda$. 
Following \cite{Kor}, for a discrete set $\Lambda\subset\bbR$ and $a>0$ we will say that \emph{$\Lambda$ has Kahane density $a$}, $D_{\rm Ka}(\Lambda)=a$, if 
\[
\int_1^\infty (\abs{n_\Lambda(t)-at}+\abs{n_{-\Lambda}(t)-at})\frac{\dd t}{t^2}<\infty.
\]
One can also consider sets with different densities of positive and negative parts, but we do not need this. 
For a discrete set $\Lambda\subset\bbR$ and for $R>0$ one says that the \emph{completeness radius of $\Lambda$ is $R$}, if 
\[
R=\sup\bigl\{r>0: \text{the linear span of } \{\ee^{\ii\lambda t}\}_{\lambda\in\Lambda} \text{ is dense in $L^2(-r,r)$.}\bigr\}
\]
We state the simplified version of the Beurling--Malliavin theorem, sufficient for our purposes. 
\begin{proposition}\cite{Be-Ma,Kahane}
\label{prp:dd1}
Suppose a discrete set $\Lambda\subset\bbR$ has Kahane density $a>0$. Then the completeness radius of $\Lambda$ is $\pi a$. 
\end{proposition}
This statement can be found in the survey \cite[Theorem 71]{Redheffer}, and the history and background are also described in \cite{Redheffer}. The full version of the Beurling--Malliavin theorem in \cite{Be-Ma} determines the completeness radius of \emph{any} discrete set and uses a more complicated notion of density.

\subsection{A probabilistic lemma}

\begin{lemma}\label{lma:dd2}
Let $\omega=\{\omega_n\}_{n\in\bbZ}$ be a sequence of Bernoulli random variables, taking value $1$ with probability $q\in(0,1)$ and value $0$ with probability $1-q$. Then almost surely $D_{\rm Ka}(\supp\omega)=q$. 
\end{lemma}
\begin{proof}
The density of the positive part of $\supp\omega$ is  
\[
n_{\supp\omega}(t)=\sum_{0<n<t}\omega_n,
\]
and so, by a simple calculation, for any $t>0$, 
\begin{align*}
\bbE\{n_{\supp\omega}(t)\}&=q(N-1),
\\
\bbE\{(n_{\supp\omega}(t)-q(N-1))^2\}&=q(1-q)(N-1),
\end{align*}
where $N$ is the largest integer such that $N<t$.  It follows that 
\[
\bbE\{\abs{n_{\supp\omega}(t)-q(N-1)}\}\leq \sqrt{q(1-q)t}.
\]
From here we find
\[
\bbE\left\{\int_1^\infty\abs{n_{\supp\omega}(t)-qt}\frac{\dd t}{t^2}\right\}
=
\int_1^\infty O(\sqrt{t})\frac{\dd t}{t^2}<\infty.
\]
Thus, almost surely we have 
\[
\int_1^\infty\abs{n_{\supp\omega}(t)-qt}\frac{\dd t}{t^2}<\infty.
\]
In the same way one considers the integral corresponding to the negative part of $\supp\omega$. We obtain that $D_{\rm Ka}(\supp\omega)=q$ almost surely, as required. 
\end{proof}

\subsection{Proof of Theorem~\ref{thm:aa4a}}
\emph{Part (i):}
We assume without loss of generality that $\varphi(k)=\varphi_{\max}$ on the arc $\abs{k}\leq a$. We denote $p=\bbP_0(\{\varkappa_{\max}\})$; by assumption \eqref{eq:aa17}, $p$ is positive. Let us denote  $\omega_n=1-\chi_{\{\varkappa_{\max}\}}(\varkappa_n)$; thus $\omega_n$ is a sequence of Bernoulli random variables taking values $1$ with probability $1-p$ and $0$ with probability $p$. By Lemma~\ref{lma:dd2} (with $q=1-p$), we have $D_{\rm Ka}(\supp\omega)=1-p$ almost surely. Let us fix a sequence $\omega$ with this property. By Proposition~\ref{prp:dd1} and the assumption $\frac{a}{\pi}>1-p$, the set of exponentials $\{\ee^{\ii nk}\}_{n\in\supp\omega}$ is \emph{incomplete} in $L^2(-a,a)$. This means that there exists a non-zero function $f\in L^2(-a,a)$ such that 
\begin{equation}
\int_{-a}^a f(k)\ee^{-\ii nk}\dd k=0, \quad \forall n\in\supp\omega.
\label{eq:dd2}
\end{equation}
Let $g$ be the extension of $f$ by zero to $(-\pi,\pi)$. Then \eqref{eq:dd2} means that $\widehat{g}_n=0$ for all $n\in\supp\omega$. Equivalently, $\supp\widehat{g}\subset\bbZ\setminus\supp\omega$. Recalling our definition of $\omega$, we see that 
\[
\supp\widehat{g}\subset \{n:\varkappa_n=\varkappa_{\max}\}.
\]
Now let us compute the quadratic form of ${\myQ}_\varkappa$ on the sequence $\widehat{g}$: 
\begin{align*}
\jap{{\myQ}_\varkappa\widehat{g},\widehat{g}}_{\ell^2}
&=
\sum_{n,m\in\bbZ}\widehat{\varphi}_{n-m}\sqrt{\varkappa_m}\widehat{g}_m\sqrt{\varkappa_n}\overline{\widehat{g}_n}
=
\varkappa_{\max}
\sum_{n,m\in\bbZ}\widehat{\varphi}_{n-m}\widehat{g}_m\overline{\widehat{g}_n}
\\
&=
\varkappa_{\max}
\int_{-\pi}^\pi \varphi(k)\abs{g(k)}^2\frac{\dd k}{2\pi}
=
\varkappa_{\max}\varphi_{\max}
\int_{-\pi}^\pi\abs{g(k)}^2\frac{\dd k}{2\pi}
\\
&=
\sigma_{\max}\norm{g}^2_{L^2}
=\sigma_{\max}\norm{\widehat{g}}^2_{\ell^2}.
\end{align*}
Since $\sigma_{\max}$ is the supremum of the spectrum, it follows that $\sigma_{\max}$ is an eigenvalue of ${\myQ}_\varkappa$ with the eigenvector $\widehat{g}$. 

We recall that the above calculation applies to all $\varkappa$ from a set of full measure. We obtain that $\sigma_{\max}$ is an eigenvalue of ${\myQ}_\varkappa$ almost surely. By Proposition~\ref{prp:A3}, it follows that $\nu(\{\sigma_{\max}\})>0$. 

\emph{Part (ii):}
Let us assume without loss of generality that  the set $\{k: \varphi(k)=\varphi_{\max}\}$ is contained in the arc $[-a,a]$. Again, let $p=\bbP_0(\{\varkappa_{\max}\})$ and define $\omega_n$ in the same way as in part (i). Take any $\omega$ with $D_{\rm Ka}(\supp\omega)=1-p$. To get a contradiction, suppose that $\sigma_{\max}$ is an eigenvalue of ${\myQ}_\varkappa$ with an eigenvector $\widehat g$, where $g\in L^2(-\pi,\pi)$. We have 
\begin{align*}
\sigma_{\max}
=
\frac{\jap{{\myQ}_\varkappa\widehat{g},\widehat{g}}_{\ell^2}}{\norm{\widehat{g}}_{\ell^2}^2}
=
\frac{\jap{L_\varphi\sqrt{\varkappa}\widehat{g},\sqrt{\varkappa}\widehat{g}}_{\ell^2}}{\norm{\widehat{g}}_{\ell^2}^2}
=
\frac{\jap{L_\varphi\sqrt{\varkappa}\widehat{g},\sqrt{\varkappa}\widehat{g}}_{\ell^2}}{\norm{\sqrt{\varkappa}\widehat{g}}_{\ell^2}^2}
\frac{\norm{\sqrt{\varkappa}\widehat{g}}^2}{\norm{\widehat{g}}_{\ell^2}^2}.
\end{align*}
For the equality to hold, both fractions in the right-hand side must attain their maximal values. This implies that $\supp\widehat{g}\subset\{n: \varkappa_n=\varkappa_{\max}\}$ and $\supp g\subset[-a,a]$. In particular, we have $\widehat{g}_n=0$ for all $n\in\supp\omega$. By Proposition~\ref{prp:dd1}, $g$ must vanish, this is a contradiction. Thus, $\sigma_{\max}$ is not an eigenvalue of ${\myQ}_\varkappa$ almost surely. By Proposition~\ref{prp:A3}, we find that $\nu(\{\sigma_{\max}\})=0$. The proof of Theorem~\ref{thm:aa4a} is complete. 
\qed

\section{Lifshitz tails bounds: proof of Theorem~\ref{thm:aa5}}
\label{sec.d} 

\subsection{The case of a model symbol $\varphi$}
We start by proving Theorem~\ref{thm:aa5} in the key particular case of the explicit model symbol 
\[
\varphi(k)=a+b\cos k
\]
with $a\in\bbR$ and $b>0$. In this case, $L_\varphi$ is the tri-diagonal matrix
with $a$ on the main diagonal and $b/2$ on the sub and super-diagonals, and 
\[
\varphi_{\max}=a+b. 
\]
For convenience of referencing, let us state the intermediate result that we are going to prove. 
\begin{theorem}\label{thm:gg1}
Assume condition \eqref{eq:aa11a} on the distribution $\bbP_0$. Let $\nu$ be the IDS measure for the operator ${\myQ}_\varkappa=\sqrt{\varkappa}L_{\varphi}\sqrt{\varkappa}$ with $\varphi(k)=a+b\cos k$ and $a+b>0$, and let $\sigma_{\max}=(a+b)\varkappa_{\max}$.  Then 
\begin{align}
\lim_{\delta\to0_+} \log \bigl(-\log \nu((\sigma_{\max}-\delta,\sigma_{\max}))\bigr)/\log\delta = -\frac12.
\label{eq:gg1} 
\end{align}
\end{theorem}
We note that in the framework of this theorem, the model symbol $a+b\cos k$ may attain negative values, since $\varphi_{\min}=a-b$ can be either positive or negative. While the sign of  $a-b$ plays no role in the proof of Theorem~\ref{thm:gg1}, it must remain unrestricted so that Theorem~\ref{thm:gg1} can be applied in the proof of Theorem~\ref{thm:aa5}. Accordingly, in the setting of Theorem~\ref{thm:gg1} we temporarily dispense with the assumption $\varphi\geq0$. In any case, the IDS measure $\nu$ of ${\myQ}_\varkappa$ is well-defined by \eqref{eq:cc0} independently of the sign of $\varphi$.

\subsection{Bracketing}
We start the proof of Theorem~\ref{thm:gg1}. 
By a trivial rescaling, we may assume $b=2$. Next, we denote 
\[
J=L_\varphi\quad\text{ for }\quad \varphi(k)=2\cos k. 
\]
In other words, $J$ is the standard two-sided Jacobi matrix with entries
\begin{equation*}
J_{n-1,n}=J_{n+1,n}=1
\end{equation*}
and $J_{n,m}=0$ otherwise. With this notation, 
\[
{\myQ}_\varkappa=\sqrt{\varkappa}(aI+J)\sqrt{\varkappa}=a\varkappa+\sqrt{\varkappa}J\sqrt{\varkappa}.
\]
We follow the approach of Kirsch and Martinelli \cite{Ki-Ma:83} as presented in \cite{Simon}.
Fix $L$ and consider the $L\times L$ matrices
\begin{equation*}
J^{(N)} =
\begin{pmatrix}
-1 & 1 &  &  &  \\
1 & 0 & 1 &  &  \\
& \ddots & \ddots & \ddots &  \\
&  & 1 & 0 & 1 \\
&  &  & 1 & -1%
\end{pmatrix}%
, \qquad J^{(D)} =
\begin{pmatrix}
1 & 1 &  &  &  \\
1 & 0 & 1 &  &  \\
& \ddots & \ddots & \ddots &  \\
&  & 1 & 0 & 1 \\
&  &  & 1 & 1%
\end{pmatrix}%
.
\end{equation*}
Here on the diagonal we have entries
\begin{align*}
-1,0,0,\dots,0,0,-1 \quad\text{ for $J^{(N)}$,} \\
1,0,0,\dots,0,0,1 \quad\text{ for $J^{(D)}$,}
\end{align*}
the entries on the sub- and super-diagonals are $1$, and all other entries
are zeros.  These matrices are the discrete analogues of the Neumann and Dirichlet boundary value problems  for $-\dd^2/\dd x^2$ on a finite interval. 

Both $J^{(N)}$ and $J^{(D)}$ are easily diagonalised; the
eigenvalues are
\begin{align*}
\lambda_k(J^{(D)})&=-2\cos(\tfrac{\pi}{L}k), \quad k=1,\dots, L, \\
\lambda_k(J^{(N)})&=-2\cos(\tfrac{\pi}{L}(k-1)), \quad k=1,\dots, L.
\end{align*}
Further, we write 
\[
\bbZ=\bigcup_{n\in\bbZ}Z_n, \quad 
Z_n:=\{nL, nL+1,\dots,nL+L-1\},
\]
and the corresponding orthogonal sum
\begin{equation}
\ell^2(\bbZ)=\bigoplus_{n\in\bbZ}\ell^2(Z_n). 
\label{eq:gg3a}
\end{equation}
Of course, $\ell^2(Z_n)$ here is isomorphic to $\bbC^L$. 
With respect to the orthogonal sum decomposition \eqref{eq:gg3a}, we have the bracketing
\begin{equation*}
\bigoplus_{n\in\bbZ} J^{(N)}_n\leq J\leq \bigoplus_{n\in\bbZ} J^{(D)}_n
\end{equation*}
where the left-hand (resp. right-hand) side is the orthogonal sum of countably
many copies of shifted versions of $J^{(N)}$ (resp. $J^{(D)}$). 
In what follows, we suppress the index $n$ in our notation, i.e. we write the previous equation as 
\begin{equation*}
\bigoplus_{n\in\bbZ} J^{(N)}\leq J\leq \bigoplus_{n\in\bbZ} J^{(D)}.
\end{equation*}
As a consequence of the last relation, we have
\begin{equation*}
\bigoplus_{n\in\bbZ} \sqrt{\varkappa} J^{(N)}\sqrt{\varkappa} 
\leq \sqrt{\varkappa}
J\sqrt{\varkappa} 
\leq \bigoplus_{n\in\bbZ} \sqrt{\varkappa} J^{(D)}\sqrt{\varkappa}.
\end{equation*}
(Here we are abusing notation a little; in the expressions $\sqrt{\varkappa}
J^{(N)}\sqrt{\varkappa}$ and $\sqrt{\varkappa} J^{(D)}\sqrt{\varkappa}$, the
symbol $\varkappa$ means a length $L$ sequence $\varkappa_1,\dots,\varkappa_L $ of i.i.d.  random variables, while in the expression $\sqrt{\varkappa} J\sqrt{\varkappa}$ it means the infinite sequence.) It follows that
\begin{equation}
\bigoplus_{n\in\bbZ}
\bigl(a\varkappa+\sqrt{\varkappa} J^{(N)}\sqrt{\varkappa}\bigr) 
\leq
{\myQ}_\varkappa 
\leq 
\bigoplus_{n\in\bbZ}
\bigl(a\varkappa+\sqrt{\varkappa}
J^{(D)}\sqrt{\varkappa}\bigr).
\label{eq:gg6}
\end{equation}
Let us denote 
\[
\myQ_\varkappa^{L,N}=a\varkappa+\sqrt{\varkappa} J^{(N)}\sqrt{\varkappa}, 
\quad
\myQ_\varkappa^{L,D}=a\varkappa+\sqrt{\varkappa} J^{(D)}\sqrt{\varkappa}.
\]
We consider $\varkappa=(\varkappa_1,\dots,\varkappa_L)$ as the random variable on $\bbR^L$ with respect to the probability measure $\bbP_L$, which is the product of $L$ copies of $\bbP_0$; we denote by $\bbE_L$ the corresponding expectation. With this notation, by a variational argument similar to Proposition~\ref{prp:A2a}, from \eqref{eq:gg6} one obtains (see \cite[Theorem~2.5]{Simon}) the upper and lower bounds 
\begin{equation}
\nu_L^{(N)}((\lambda,\infty))
\leq
\nu((\lambda,\infty))
\leq
\nu_L^{(D)}((\lambda,\infty)), \quad \lambda>0,
\label{eq:gg7}
\end{equation}
where
\begin{equation}
\nu_L^{(\#)}((\lambda,\infty))
=
L^{-1} \bbE_L\left\{\calN(\lambda;\myQ_\varkappa^{L,\#})\right\}
\label{eq:gg8}
\end{equation}
with $\#=N$ or $D$, and the counting function $\calN$ defined in \eqref{eq:CF}.

\subsection{Lower bound}
Combining \eqref{eq:gg7} and \eqref{eq:gg8}, we find 
\begin{equation}
\nu((\sigma_{\max}-\delta,\infty))
\geq
L^{-1} \bbE_L\left\{\calN(\sigma_{\max}-\delta;\myQ_\varkappa^{L,N})\right\}.
\label{eq:gg11}
\end{equation}
We need to give a lower bound for the eigenvalue counting function in the right-hand side. 
We will do this by proving that if 
\begin{equation}
\varkappa_n\geq \varkappa_{\max}-\frac{\delta}{2a+4}\quad \text{ for all $n=1,\dots,L$,}
\label{eq:gg9}
\end{equation}
then $\myQ_\varkappa^{L,N}$ has at least one eigenvalue greater than $\sigma_{\max}-\delta$, whenever $\delta$ and $L$ are linked by 
\begin{equation}
\frac{2\pi\sqrt{\varkappa_{\max}}}{\sqrt{\delta}}
<L\leq
1+\frac{2\pi\sqrt{\varkappa_{\max}}}{\sqrt{\delta}}.
\label{eq:gg10}
\end{equation}

Observe that the top Neumann eigenvalue of $J^{(N)}$ is
\begin{equation*}
\lambda_L(J^{(N)})=-2\cos\pi(1-L^{-1})=:2-\mu_L,
\quad\text{ where }\quad
\mu_L=\frac{\pi^2}{L^2}+o(L^{-2}).
\end{equation*}
If \eqref{eq:gg9} is satisfied, then by simple variational considerations for the top eigenvalue of $\myQ_\varkappa^{L,N}$ we find 
\begin{align*}
\lambda_L(\myQ_\varkappa^{L,N})
&\geq
(\varkappa_{\max}-\delta/(2a+4))(a+\lambda_L(J^{(N)}))
\\
&=
(\varkappa_{\max}-\delta/(2a+4))(a+2-\mu_L)
\geq
\sigma_{\max}-\varkappa_{\max}\mu_L-\delta/2. 
\end{align*}
Now an elementary calculation shows that if $\delta$ and $L$ are linked by \eqref{eq:gg10}, then 
\[
\lambda_L(\myQ_\varkappa^{L,N})\geq \sigma_{\max}-\delta
\]
for all sufficiently small $\delta$. 
We can now give a lower bound for the right-hand side of \eqref{eq:gg11} by estimating the $\bbP_L$-probability of the event \eqref{eq:gg9}. Using our assumption \eqref{eq:aa11a}, we find
\begin{align*}
\bbP_L\bigl(\text{\eqref{eq:gg9}}\bigr)
=
\biggl(\bbP_0\bigl(\{\varkappa_0\geq\varkappa_{\max}-\tfrac{\delta}{2a+4}\}\bigr)\biggr)^L
\geq
\left(C(\delta/(2a+4))^\ell\right)^L.
\end{align*}
From here by \eqref{eq:gg11} we find 
\[
\nu((\sigma_{\max}-\delta,\infty))
\geq
L^{-1}\left(C(\delta/(2a+4))^\ell\right)^L. 
\]
Recalling \eqref{eq:gg10}, we rewrite this as 
\[
\nu((\sigma_{\max}-\delta,\infty))
\geq
C_1\sqrt{\delta}(C_2\delta^\ell)^{C_3/\sqrt{\delta}}.
\]
Taking logarithms, from here we find
\begin{equation*}
\liminf_{\delta\to0_+} \log\bigl(-\log \nu(\sigma_{\max}-\delta,\sigma_{\max})\bigr)/\log\delta \geq -\frac12,
\end{equation*}
i.e. we obtain the lower bound in \eqref{eq:gg1}.

\subsection{Upper bound}
Combining \eqref{eq:gg7} and \eqref{eq:gg8}, we find 
\begin{equation}
\nu((\sigma_{\max}-\delta,\infty))
\leq
L^{-1} \bbE_L\left\{\calN(\sigma_{\max}-\delta;\myQ_\varkappa^{L,D})\right\},
\label{eq:gg12}
\end{equation}
and so we need to give an upper bound for the eigenvalue counting function in the right-hand side. 

We start by giving an estimate which does the same job for us as Temple's
inequality does for Simon in \cite{Simon}. Observe that the top Dirichlet
eigenvalue is $2$, with the normalized eigenvector
\begin{equation*}
\mathbbm{1}_L:=L^{-1/2}(1,\dots,1)^T\in{\mathbb{C}}^L.
\end{equation*}
The next Dirichlet eigenvalue is
\begin{equation*}
\lambda_{L-1}(J^{(D)})=-2\cos\pi(1-L^{-1})=2-\mu_L, 
\quad\text{ where }\quad
\mu_L=\frac{\pi^2}{L^2}+o(L^{-2}).
\end{equation*}
It follows that we have the bound
\begin{align*}
J^{(D)}&\leq \lambda_{L-1}(J^{(D)})\bigl(I-\langle \cdot,\mathbbm{1}_L\rangle
\mathbbm{1}_L\bigr) +2\langle \cdot,\mathbbm{1}_L\rangle\mathbbm{1}_L
\\
&=2-\mu_L+\mu_L \langle \cdot,\mathbbm{1}_L\rangle\mathbbm{1}_L.
\end{align*}
From here we get
\begin{equation*}
a\varkappa+\sqrt{\varkappa} J^{(D)} \sqrt{\varkappa} \leq
(a+2-\mu_L)\varkappa+\mu_L\langle \cdot,\sqrt{\varkappa}\mathbbm{1}%
_L\rangle\sqrt{\varkappa}\mathbbm{1}_L.
\end{equation*}
Using the triangle inequality, we obtain the norm bound
\begin{align*}
\norm{a\varkappa+\sqrt{\varkappa} J^{(D)} \sqrt{\varkappa}}
&\leq
(a+2-\mu_L)\kappamax+\mu_L
\norm{\sqrt{\varkappa}\mathbbm{1}_L}^2 
\\
&=\sigma_{\max} -\mu_L\kappamax+\mu_L\jap{\varkappa},
\end{align*}
where
\begin{equation*}
\jap{\varkappa}=\frac1L\sum_{k=1}^L\varkappa_k.
\end{equation*}

Let us rephrase this as follows. Fix some $0<\gamma<1$. Suppose $\varkappa$ satisfies
\begin{equation}
\jap{\varkappa}\leq \gamma\kappamax.  \label{gg5}
\end{equation}
Then for this $\varkappa$ we have
\begin{equation}
\norm{a\varkappa+\sqrt{\varkappa} J^{(D)} \sqrt{\varkappa}} 
\leq
\sigma_{\max}-(1-\gamma)\kappamax\mu_L.
\label{eq:gg13}
\end{equation}

Next, for any small $\delta>0$, take $L$ such that $\delta$ and $L$ are linked by 
\[
\frac{\pi\sqrt{(1-\gamma)\varkappa_{\max}}}{2\sqrt{\delta}}
<L\leq1+\frac{\pi\sqrt{(1-\gamma)\varkappa_{\max}}}{2\sqrt{\delta}}. 
\]
Then it is easy to check that for all sufficiently small $\delta$, the right-hand side of \eqref{eq:gg13} is less than $\sigma_{\max}-\delta$. Thus, $\varkappa$ satisfying \eqref{gg5} do not contribute to the right-hand side of \eqref{eq:gg12}.

Now using our assumption that the support of $\bbP_0$ is not a single point, we find that 
$\bbE_0\{\varkappa_0\}<\varkappa_{\max}$. Thus, we can find $\gamma\in(0,1)$ such that 
\[
\bbE_0\{\varkappa_0\}<\gamma\varkappa_{\max}.
\]
Then by a large deviations argument (see e.g. \cite[Theorem~4.2]{Simon}), the probability of the set of all $\varkappa$ that do NOT satisfy \eqref{gg5}, is $O(\ee^{-C_\gamma L})$ with some $C_\gamma>0$. Thus, estimating the right-hand side of \eqref{eq:gg12}, we obtain
\begin{equation*}
\nu((\sigma_{\max}-\delta,\infty))\leq \ee^{-C_\gamma L}=\ee^{-C^{\prime
}_\gamma/\sqrt{\delta}}.
\end{equation*}
From here we find
\begin{equation*}
\limsup_{\delta\to0_+} \log\bigl(-\log \nu((\sigma_{\max}-\delta,\sigma_{\max}))\bigr)/\log\delta \leq -\frac12,
\end{equation*}
i.e. we obtain the upper bound in \eqref{eq:gg1}. The proof of Theorem~\ref{thm:gg1} is complete. \qed

\subsection{Comparison lemma}
In the rest of the section we show how to derive Theorem~\ref{thm:aa5} from Theorem~\ref{thm:gg1}. We follow the idea of \cite{GRM1}, comparing the symbol $\varphi$ with the rescaled symbol $\cos k$. Here we deal with operators ${\myQ}_\varkappa$ corresponding to several symbols, and so we will use notation ${\myQ}_{\varphi,\varkappa}$. Similarly, we will write $\nu(\Delta;\varphi)$ instead of $\nu(\Delta)$ (the sequence $\varkappa$ will be always the same.) 

We start by isolating two simple statements. 
\begin{lemma}\label{lma:d1}
Let $\varphi_1$ and $\varphi_2$ be two non-constant real-valued symbols in $L^\infty(\bbT)$ and let $\varkappa$ be a sequence of random variables satisfying  Assumption~\ref{ass:1}(a). 
\begin{enumerate}[\rm (i)]
\item
Suppose $\varphi_1(k)=\varphi_2(k-k_0)$ for some $k_0$. Then the IDS measures $\nu(\cdot;\varphi_1)$ and $\nu(\cdot;\varphi_2)$ coincide. 
\item
Suppose $\varphi_1\leq\varphi_2$ on $\bbT$. Then for any $\lambda>0$, 
\[
\nu((\lambda,\infty);\varphi_1)\leq \nu((\lambda,\infty);\varphi_2).
\]
\end{enumerate}
\end{lemma}
\begin{proof}
(i) It suffices to note that the shift of the variable $k$ by a constant $k_0$ results in the conjugation of ${\myQ}_\varkappa$ by the unitary operator of multiplication by the oscillating sequence $\ee^{-\ii n k_0}$. It is clear from the definition of $\nu$ that this conjugation does not change $\nu$. 

(ii) From $\varphi_1\leq \varphi_2$ we immediately obtain the operator inequalities (in the sense of quadratic forms)
\[
L_{\varphi_1}\leq L_{\varphi_2}.
\]
Multiplying by $\sqrt{\varkappa}$ on both sides, we find
\[
{\myQ}_{\varphi_1,\varkappa}\leq {\myQ}_{\varphi_2,\varkappa}.
\]
Now the required estimate follows from Proposition~\ref{prp:A2a} of the Appendix. 
\end{proof}

\subsection{Proof of Theorem~\ref{thm:aa5}}
We have $\varphi(k_{\max})=\varphi_{\max}$ at the maximum; by our assumptions, $\varphi''(k_{\max})<0$. It follows that 
\[
c_-(1-\cos k)
\leq \varphi_{\max}-\varphi(k-k_{\max})
\leq c_+ (1-\cos k),
\]
with some $0<c_-<c_+$. We rewrite this as 
\[
\varphi_{-}(k+k_{\max})\leq \varphi(k)\leq \varphi_{+}(k+k_{\max}),
\]
with 
\begin{align*}
\varphi_{-}&=
\varphi_{\max}-c_++c_+\cos k,
\\
\varphi_{+}&=
\varphi_{\max}-c_-+c_-\cos k.
\end{align*}
By Lemma~\ref{lma:d1}, we find 
\[
\nu((\sigma_{\max}-\delta,\infty);\varphi_{-})
\leq
\nu((\sigma_{\max}-\delta,\infty);\varphi)
\leq
\nu((\sigma_{\max}-\delta,\infty);\varphi_{+}).
\]
We note that 
\[
\max\varphi_{-}=\max\varphi_{+}=\varphi_{\max}.
\]
Finally, we note that both symbols $\varphi_{-}$, $\varphi_{+}$, are of the form $a+b\cos k$ considered in Theorem~\ref{thm:gg1}. Thus, the required estimates  follow from this theorem. The proof of Theorem~\ref{thm:aa5} is complete. \qed

\section{Anderson type localisation: proof of Theorem~\ref{thm:aa6}}
\label{sec.f} 

\subsection{Overview and notation}
As in the previous section, we use the realisations 
\[
\calP_\varkappa =L_{\theta}\varkappa L_{\theta}
\quad \text{ and }\quad
{\myQ}_\varkappa=\sqrt{\varkappa}L_\varphi\sqrt{\varkappa}
\]
in $\ell^2(\bbZ)$; recall that $\theta=\sqrt{\varphi}$ and $(L_\theta)^2=L_\varphi$.  Our key idea is to write 
\begin{equation}
\varphi=1+\phi
\quad\text{ and }\quad
{\myQ}_\varkappa=\varkappa+\sqrt{\varkappa}L_\phi\sqrt{\varkappa}
\label{eq:f11}
\end{equation}
and treat $\varkappa$ as the ``main'' diagonal term and $\sqrt{\varkappa}L_\phi\sqrt{\varkappa}$ as a small off-diagonal perturbation. We follow very closely the argument of \cite{Aiz-Mol}; in their terminology, our asymptotic regime of small $\phi$ corresponds to \emph{high disorder}. The argument consists of the following parts:
\begin{itemize}
\item
Connection between the resolvents of $\calP_\varkappa$ and ${\myQ}_\varkappa$;
\item
Rank-one resolvent identities; 
\item
Decoupling inequalities;
\item
Simon--Wolff criterion.
\end{itemize}
The first item is specific to our model, while the other three items are exactly as in \cite{Aiz-Mol}. In what follows, $z$ is a spectral parameter with $\Im z>0$. We denote the resolvents of $\calP_\varkappa$, ${\myQ}_\varkappa$ by 
\[
R(z)=(\calP_\varkappa-z)^{-1}, \quad
\widetilde{R}(z)=({\myQ}_\varkappa-z)^{-1}.
\]
For $n,m\in\bbZ$, we denote 
\begin{equation}\label{eq:ress}
R_{n,m}(z)=\jap{(\calP_\varkappa-z)^{-1}\psi_n,\psi_m},
\quad
\widetilde{R}_{n,m}(z)=\jap{({\myQ}_\varkappa-z)^{-1}\delta_n,\delta_m}. 
\end{equation}
Because of the shift invariance in Assumption~\ref{ass:1}(c), we are free to shift $n$ and $m$ by the same integer; it will be notationally convenient to take one of these indices equal to zero, so we will be dealing mostly with $R_{0,m}(z)$ and $\widetilde{R}_{0,m}(z)$. 

Most of the proof hinges on estimates on $R_{0,m}(z)$ and $\widetilde{R}_{0,m}(z)$. The dependence on $\varkappa$ is much more transparent (linear) in $R(z)$. On the other hand, $\widetilde{R}(z)$ can be analysed in terms of our splitting \eqref{eq:f11}. Thus, we will be combining information coming from $R(z)$ and $\widetilde{R}(z)$. Of course, the two resolvents are explicitly related and we start by clarifying this relation. 

\subsection{Connection between resolvents $R(z)$ and $\widetilde{R}(z)$.}

\begin{lemma}\label{lma:f1}
For any $m\in\bbZ$, we have for $R_{0,m}(z)$ and $\widetilde{R}_{0,m}(z)$ of \eqref{eq:ress}
\begin{equation}
\delta_{0,m}+z\widetilde{R}_{0,m}(z)
=
\sqrt{\varkappa_0\varkappa_m}R_{0,m}(z).
\label{eq:f1}
\end{equation}
\end{lemma}
\begin{proof}
By \eqref{eq:e9}, by taking linear combinations, we find
\[
{\myQ}_\varkappa f({\myQ}_\varkappa)
=
\sqrt{\varkappa}L_{\theta}f(\calP_\varkappa) L_{\theta}\sqrt{\varkappa}
\]
for any polynomial $f$. By a limiting argument, this extends to all continuous functions $f$. In particular, for $f(x)=(x-z)^{-1}$ we obtain 
\[
I+z({\myQ}_\varkappa-z)^{-1}
=
{\myQ}_\varkappa({\myQ}_\varkappa-z)^{-1}
=
\sqrt{\varkappa}L_{\theta}(\calP_\varkappa-z)^{-1} L_{\theta}\sqrt{\varkappa}.
\]
Applying this to $\delta_0$, evaluating the inner product with $\delta_m$ and recalling that $\psi_m=L_{\theta}\delta_m$, we obtain the required identity. 
\end{proof}

\subsection{Rank-one resolvent identities}
We first recall the standard rank-one resolvent identity. Writing 
\[
\calP_\varkappa=\varkappa_n\jap{\bigcdot,\psi_n}\psi_n+\calP_\varkappa^{\neq}, 
\qquad 
\calP_\varkappa^{\neq}=\sum_{k\not=n}\varkappa_k\jap{\bigcdot,\psi_k}\psi_k
\]
for a fixed $n\in\bbZ$ and 
\[
{R}^{\neq}(z)=({\calP}_\varkappa^{\neq}-z)^{-1},
\]
we have 
\begin{equation}
{R}(z)
=
{R}^{\neq}(z)
-
\frac{\varkappa_n\jap{\bigcdot,{R}^{\neq}(\overline{z})\psi_n}{R}^{\neq}(z)\psi_n}{1+\varkappa_n\jap{{R}^{\neq}(z)\psi_n,\psi_n}}.
\label{eq:f3}
\end{equation}

\begin{lemma}\label{lma:f2}
For any $n\in\bbZ$ ($n=0$ is not excluded) the dependence of $\widetilde{R}_{0,0}(z)$ on $\varkappa_n$ can be written as
\begin{equation}
\widetilde{R}_{0,0}(z)
=
\frac{a\varkappa_n-\beta}{\varkappa_n-\alpha},
\label{eq:f2}
\end{equation}
where the complex numbers $a,\alpha,\beta$ may depend on $z$ and on $\{\varkappa_k\}_{k\not=n}$. 
\end{lemma}
\begin{proof}
Taking $m=0$ in \eqref{eq:f1}, we find
\[
1+z\widetilde{R}_{0,0}(z)
=
\varkappa_0R_{0,0}(z).
\]
Using \eqref{eq:f3} and denoting for brevity
\[
{R}^{\neq}_{i,j}(z)=\jap{{R}^{\neq}(z)\psi_i,\psi_j},
\]
we find 
\begin{equation}
1+z\widetilde{R}_{0,0}(z)
=
\varkappa_0\left\{ {R}^{\neq}_{0,0}(z)-\frac{\varkappa_n R^{\neq}_{0,n}(z)R^{\neq}_{n,0}(z)}{1+\varkappa_n R^{\neq}_{n,n}(z)}\right\}.
\label{eq:f4}
\end{equation}
We note that for $\Im z>0$ the complex number $R^{\neq}_{n,n}(z)$ appearing in the denominator is non-zero, because it has a strictly positive imaginary part. If $n\not=0$, the relation \eqref{eq:f4} immediately yields the representation \eqref{eq:f2}. If $n=0$, we need to notice a cancellation in the numerator, which yields
\[
1+z\widetilde{R}_{0,0}(z)=\frac{\varkappa_0 R^{\neq}_{0,0}(z)}{1+\varkappa_0 R^{\neq}_{0,0}(z)};
\]
this is again in the required form \eqref{eq:f2}. 
\end{proof}

\begin{lemma}\label{lma:f3}
Consider $\widetilde{R}_{0,m}(z)$ for $m\not=0$.
\begin{enumerate}[\rm (i)]
\item
Let $n\not=0$ and $n\not=m$. Then the dependence of $\widetilde{R}_{0,m}(z)$ on $\varkappa_n$ can be written as
\begin{equation}
\widetilde{R}_{0,m}(z)
=
\frac{\gamma}{\varkappa_n-\alpha},
\label{eq:f5}
\end{equation}
where the complex numbers $\alpha,\gamma$ may depend on $z$ and on $\{\varkappa_k\}_{k\not=n}$. 
\item
The dependence of $\widetilde{R}_{0,m}(z)$ on $\varkappa_0$ can be written as
\begin{equation}
\widetilde{R}_{0,m}(z)
=
\frac{\gamma\sqrt{\varkappa_0}}{\varkappa_0-\alpha},
\label{eq:f6}
\end{equation}
where the complex numbers $\alpha,\gamma$ may depend on $z$ and on $\{\varkappa_k\}_{k\not=0}$. 
\end{enumerate}
\end{lemma}
\begin{proof}
Combining \eqref{eq:f1} and \eqref{eq:f3}, we find
\[
z\widetilde{R}_{0,m}(z)
=
\sqrt{\varkappa_0\varkappa_m}\left\{
R_{0,m}^{\neq}-\frac{\varkappa_nR_{0,n}^{\neq}(z)R_{n,m}^{\neq}(z)}{1+\varkappa_nR_{n,n}^{\neq}(z)}
\right\}.
\]
Again, $R_{n,n}^{\neq}(z)\not=0$ in the denominator. From here we immediately obtain \eqref{eq:f5}, \eqref{eq:f6}. 
\end{proof}

\subsection{Decoupling inequalities}
The following statement is \cite[Theorem~8.8]{Aiz-War}, with minor modifications. 
\begin{lemma}
Assume that the distribution $\bbP_0$ satisfies \eqref{eq:g}, and let $s\in(0,\tau/2)$.
There exists a constant $A$ which depends only on $\bbP_0$ and on $s$, such that for any complex numbers $a$, $\alpha$, $\beta$, $z$, we have 
\begin{align}
\int_\bbR\Abs{\frac{a\kappa-\beta}{\kappa-\alpha}}^s\frac{1}{\abs{\kappa-z}^s}\dd\bbP_0(\kappa)
&\leq
A
\int_\bbR\Abs{\frac{a\kappa-\beta}{\kappa-\alpha}}^s\dd\bbP_0(\kappa)
\int_\bbR\frac{1}{\abs{\kappa-z}^s}\dd\bbP_0(\kappa),
\label{eq:f7}
\\
\int_\bbR\frac{\kappa^{s/2}}{\abs{\kappa-\alpha}^s}\frac{1}{\abs{\kappa-z}^s}\dd\bbP_0(\kappa)
&\leq
A
\int_\bbR\frac{\kappa^{s/2}}{\abs{\kappa-\alpha}^s}\dd\bbP_0(\kappa)
\int_\bbR\frac{1}{\abs{\kappa-z}^s}\dd\bbP_0(\kappa).
\label{eq:f8}
\end{align}
\end{lemma}
We note that all integrals here are in fact over $[\varkappa_{\min},\varkappa_{\max}]$. 
\begin{proof}
Relation \eqref{eq:f7} for $a=1$ is exactly the statement of \cite[Theorem~8.8]{Aiz-War} (in fact, the statement from \cite{Aiz-War} is more general as it applies to a more general class of measures $\bbP_0$). Of course, \eqref{eq:f7} for any $a\not=0$ immediately reduces to the case $a=1$. 

If $a=0$, one can prove \eqref{eq:f7} by following the steps of the proof in \cite{Aiz-War}, but it is quicker to deduce it from the already proven case. It suffices to take $\beta=1$; we have:
\begin{align*}
\int_\bbR\frac{1}{\abs{\kappa-\alpha}^s}\frac{1}{\abs{\kappa-z}^s}\dd\bbP_0(\kappa)
&\leq
\frac1{\varkappa_{\min}^s}
\int_\bbR\frac{\kappa^s}{\abs{\kappa-\alpha}^s}\frac{1}{\abs{\kappa-z}^s}\dd\bbP_0(\kappa)
\\
&\leq
A
\frac{1}{\varkappa_{\min}^s}
\int_\bbR\frac{\kappa^s}{\abs{\kappa-\alpha}^s}\dd\bbP_0(\kappa)
\int_\bbR\frac{1}{\abs{\kappa-z}^s}\dd\bbP_0(\kappa)
\\
&\leq
A\frac{\varkappa_{\max}^s}{\varkappa_{\min}^s}
\int_\bbR\frac{1}{\abs{\kappa-\alpha}^s}\dd\bbP_0(\kappa)
\int_\bbR\frac{1}{\abs{\kappa-z}^s}\dd\bbP_0(\kappa).
\end{align*}
In a similar way one proves \eqref{eq:f8} by deducing it from the already proven cases through ``inserting'' an extra factor of $\kappa^{s/2}$ into the integral. 
\end{proof}

This and Lemmas~\ref{lma:f2} and \ref{lma:f3} yield the following. 

\begin{lemma}\label{lma:f5}
For any $s\in(0,\tau/2)$ there exists a constant $A$ that depends only on $\bbP_0$ and on $s$ such that for any $n,m\in\bbZ$
\begin{equation}
\bbE\{\abs{\varkappa_n-z}^{-s}\abs{\widetilde{R}_{0,m}(z)}^s\}
\leq
A\, 
\bbE\{\abs{\varkappa_n-z}^{-s}\}\, 
\bbE\{\abs{\widetilde{R}_{0,m}(z)}^s\}.
\label{eq:f9}
\end{equation}
\end{lemma}
\begin{proof}
Suppose $m=0$; by Lemma~\ref{lma:f2} we have 
\[
\bbE\{\abs{\varkappa_n-z}^{-s}\abs{\widetilde{R}_{0,0}(z)}^s\}
=
\bbE\left\{\abs{\varkappa_n-z}^{-s}\Abs{\frac{a\varkappa_n-\beta}{\varkappa_n-\alpha}}^s\right\},
\]
where $a$, $\alpha$, $\beta$ depend on $z$ and on $\{\varkappa_k\}_{k\not=n}$. Taking the expectation with respect to $\varkappa_n$ and using \eqref{eq:f7}, and subsequently taking the expectation with respect to the rest of the variables, we obtain 
\begin{align*}
\bbE\left\{\abs{\varkappa_n-z}^{-s}\Abs{\frac{a\varkappa_n-\beta}{\varkappa_n-\alpha}}^s\right\}
&\leq
A\, 
\bbE\left\{\abs{\varkappa_n-z}^{-s}\right\}\, 
\bbE\left\{\Abs{\frac{a\varkappa_n-\beta}{\varkappa_n-\alpha}}^s\right\}
\\
&=
A\, \bbE\left\{\abs{\varkappa_n-z}^{-s}\right\}\, 
\bbE\{\abs{\widetilde{R}_{0,0}(z)}^s\},
\end{align*}
which proves \eqref{eq:f9} for $m=0$. The case $m\not=0$ is considered similarly by using Lemma~\ref{lma:f3} in place of Lemma~\ref{lma:f2}. 
\end{proof}
\begin{remark*}
The quantity 
\begin{equation}
D_s(z):=\bbE\{\abs{\varkappa_n-z}^{-s}\}
\label{eq:f16}
\end{equation}
can be bounded uniformly in $z\in\bbC$ by a constant that depends only on $s$ and $\bbP_0$. Thus, this quantity can be absorbed into the constant $A$ in the above lemma if desired. 
\end{remark*}

\subsection{The core of the proof: finiteness of the fractional moments}
In what follows for definiteness we take $s=\tau/4$ in Lemma~\ref{lma:f5}. 

\begin{lemma}\label{lma:f6}
Under the hypothesis of Theorem~\ref{thm:aa6}, there exists $\eta>0$ such that for any $\varphi$ satisfying \eqref{eq:g1}, we have 
\begin{equation}
\sup_{\Im z>0}\bbE\left\{\sum_{m\in\bbZ}\abs{\widetilde{R}_{0,m}(z)}^{\tau/4}\right\}<\infty.
\label{eq:f10}
\end{equation}
\end{lemma}
We note that by the shift invariance, we can replace the index ``0" in \eqref{eq:f10} by any integer. 
\begin{proof}
We start with the resolvent identity for the decomposition \eqref{eq:f11}: 
\[
\widetilde{R}(z)=(\varkappa-z)^{-1}-(\varkappa-z)^{-1}\sqrt{\varkappa}L_\phi\sqrt{\varkappa}\widetilde{R}(z). 
\]
Writing this in terms of the matrix elements of operators, for the matrix element with indices $0,n$ we obtain 
\begin{align*}
\widetilde{R}_{0,n}(z)
=\delta_{0,n}(\varkappa_0-z)^{-1}-\sum_{m\not=n} (\varkappa_n-z)^{-1}\sqrt{\varkappa_n}\widehat{\phi}_{n-m}\sqrt{\varkappa_m}\widetilde{R}_{0,m}(z). 
\end{align*}
We recall that by our assumptions, $\phi\in C^\infty$ and so the series converges absolutely due to the fast decay of the Fourier coefficients $\widehat{\phi}_m$. Using $\abs{a+b}^s\leq\abs{a}^s+\abs{b}^s$, $s=\tau/4$, from here we obtain 
\begin{align*}
\abs{\widetilde{R}_{0,n}(z)}^{\tau/4}
\leq
\delta_{0,n}\abs{\varkappa_0-z}^{-\tau/4}
+
\varkappa_{\max}^{\tau/4}\sum_{m\not=n} \abs{\widehat{\phi}_{n-m}}^{\tau/4} \abs{\varkappa_n-z}^{-\tau/4}\abs{\widetilde{R}_{0,m}(z)}^{\tau/4}. 
\end{align*}
Taking expectations, using Lemma~\ref{lma:f5} and the notation $D_s(z)$ (see \eqref{eq:f16}), we find 
\begin{align*}
\bbE\{\abs{\widetilde{R}_{0,n}(z)}^{\tau/4}\}
\leq
\delta_{0,n}D_{\tau/4}(z)
+
\varkappa_{\max}^{\tau/4}A D_{\tau/4}(z) \sum_{m\not=n} \abs{\widehat{\phi}_{n-m}}^{\tau/4} \bbE\{\abs{\widetilde{R}_{0,m}(z)}^{\tau/4}\}. 
\end{align*}
Summing over $n$, we find 
\begin{align*}
\sum_{n\in\bbZ}\bbE\{\abs{\widetilde{R}_{0,n}(z)}^{\tau/4}\}
\leq
D_{\tau/4}(z)
+
\varkappa_{\max}^{\tau/4}A D_{\tau/4}(z) \sum_{m\in\bbZ}\bbE\{\abs{\widetilde{R}_{0,m}(z)}^{\tau/4}\} \sum_{n\not=m} \abs{\widehat{\phi}_{n-m}}^{\tau/4} . 
\end{align*}
Rearranging, we find 
\begin{equation}
\left(1-\varkappa_{\max}^{\tau/4}A D_{\tau/4}(z)\sum_{k\in\bbZ} \abs{\widehat{\phi}_{k}}^{\tau/4}\right)
\sum_{n\in\bbZ}\bbE\{\abs{\widetilde{R}_{0,n}(z)}^{\tau/4}\}
\leq
D_{\tau/4}(z).
\label{eq:f15}
\end{equation}
Recall that $\sup_{z\in\bbC}D_{\tau/4}(z)<\infty$.  Denote 
\[
\eta=\frac1{2\varkappa_{\max}^{\tau/4}A \sup_{z\in\bbC}D_{\tau/4}(z)}; 
\]
then for all $\varphi$ satisfying $\sum\abs{\widehat{\phi}_{k}}^{\tau/4}<\eta$ the expression in parentheses on the left-hand side of \eqref{eq:f15} remains positive, which yields the desired estimate \eqref{eq:f10}. 
\end{proof}

\subsection{The Simon--Wolff criterion}
As in \cite{Aiz-Mol}, we rely on the Simon--Wolff criterion, which we state here for clarity. Let $A$ be a self-adjoint operator in a Hilbert space, and let $\psi$ be a cyclic element of $A$. Assume that for Lebesgue-a.e. $\lambda$ in some interval $\Delta$, we have 
\[
\lim_{\eps\to0_+}\norm{(A-\lambda-\ii\eps)^{-1}\psi}^2<\infty
\]
(the limit always exists because the norm increases as $\eps\to0_+$). Then for a.e. $\kappa\in\bbR$, the operator $A+\kappa\jap{\bigcdot,\psi}\psi$ has only point spectrum in the interval $\Delta$. 

The assumption of cyclicity of $\psi$ is not necessary here. If $\psi$ is not cyclic, then the conclusion applies to the restriction of $A+\kappa\jap{\bigcdot,\psi}\psi$ onto the cyclic subspace generated by $\psi$. (We recall that the following two subspaces coincide: (i) the cyclic subspace of $A$, generated by $\psi$ and (ii) the cyclic subspace of $A+\kappa\jap{\bigcdot,\psi}\psi$ generated by $\psi$.)

\subsection{The endgame: completing the proof}

We fix $\lambda_{\max}>0$ such that the interval $[0,\lambda_{\max}]$ contains the deterministic spectrum of $\calP_\varkappa$ for all sufficiently small $\alpha$. For example, we can take $\lambda_{\max}=2\varkappa_{\max}$. 
\begin{lemma}
For all $\varphi$ as in Lemma~\ref{lma:f6} and all $n\in\bbZ$ we have 
\begin{equation}
\sup_{0<\lambda<\lambda_{\max}}\sup_{0<\eps<1}
\bbE\left\{\norm{R(\lambda+\ii\eps)\psi_n}^{\tau/4}\right\}<\infty.
\label{eq:f12}
\end{equation}
\end{lemma}
\begin{proof}
We start with \eqref{eq:f10}. Using $(\abs{a}^2+\abs{b}^2)^{\tau/8}\leq \abs{a}^{\tau/4}+\abs{b}^{\tau/4}$, we find 
\[
\sup_{\Im z>0}\bbE\left\{\left(\sum_{m\in\bbZ}\abs{\widetilde{R}_{0,m}(z)}^2\right)^{\tau/8}\right\}<\infty.
\]
Recalling the definition of $\widetilde{R}_{0,m}(z)$ in \eqref{eq:ress}, we can rewrite this as 
\begin{equation}
\sup_{\Im z>0}\bbE\left\{\norm{\widetilde{R}(z)\delta_0}^{\tau/4}\right\}<\infty.
\label{eq:f14}
\end{equation}
Now we need to pass from estimates on $\widetilde{R}(z)$ to estimates on $R(z)$. We apply the identity \eqref{eq:f1} to $\delta_0$:
\[
\delta_0+z\widetilde{R}(z)\delta_0=\sqrt{\varkappa}_0 \sqrt{\varkappa}L_{\theta}R(z)\psi_0.
\]
We recall that $\theta=\sqrt{\varphi}=\sqrt{1+\alpha\phi}$ and in particular for all sufficiently small $\abs{\alpha}$ we have $\theta\geq1/2$, which implies that $L_{\theta}$ has a bounded inverse with $\norm{L_{\theta}^{-1}}\leq 2$. This yields
\begin{align*}
\norm{R(z)\psi_0}
&\leq 
\frac{2}{\varkappa_{\min}}\norm{\sqrt{\varkappa}L_{\theta}R(z)\psi_0}
=
\frac{2}{\varkappa_{\min}}\norm{\delta_0+z\widetilde{R}(z)\delta_0}
\\
&\leq
\frac{2}{\varkappa_{\min}}\left(1+\abs{z}\norm{\widetilde{R}(z)\delta_0}\right),
\end{align*}
and therefore
\[
\norm{R(z)\psi_0}^{\tau/4}\leq C(1+\abs{z}^{\tau/4}\norm{\widetilde{R}(z)\delta_0}^{\tau/4}).
\]
Taking expectation and using \eqref{eq:f14}, we find 
\[
\sup_{\Im z >0}\bbE\left\{\norm{R(z)\psi_0}^{\tau/4}\right\}<\infty.
\]
Of course, $\psi_0$ here can be replaced by $\psi_n$ with any $n\in\bbZ$, and we obtain the desired statement \eqref{eq:f12}. 
\end{proof}

\begin{proof}[Proof of Theorem~\ref{thm:aa6}]
We write \eqref{eq:f12} as
\[
\sup_{0<\eps<1}
\bbE\left\{\norm{R(\lambda+\ii\eps)\psi_n}^{\tau/4}\right\}\leq C,
\]
where $C$ is finite and independent on $\lambda\in(0,\lambda_{\max})$ or on $n$. Since $\norm{R(\lambda+\ii\eps)\psi_n}$ is a decreasing function of $\eps>0$, we can write this relation as 
\[
\lim_{\eps\to0_+}
\bbE\left\{\norm{R(\lambda+\ii\eps)\psi_n}^{\tau/4}\right\}\leq C.
\]
Now by Fatou's lemma, this yields 
\[
\bbE\left\{\lim_{\eps\to0_+}\norm{R(\lambda+\ii\eps)\psi_n}^{\tau/4}\right\}\leq C.
\]
In particular, this means that for every $\lambda\in(0,\lambda_{\max})$ and a.s. in $\varkappa$, the limit 
\[
\lim_{\eps\to0_+}\norm{R(\lambda+\ii\eps)\psi_n}
\]
is finite. Applying Fubini's theorem, we see that a.s. in $\varkappa$ and for Lebesgue a.e. $\lambda\in(0,\lambda_{\max})$, the limit is finite. By the Simon--Wolff criterion, this implies that for every $n\in\bbZ$ and a.s. in $\varkappa$, the spectrum of $\calP_\varkappa$ in the cyclic subspace generated by $\psi_n$ is pure point. Since the elements $\{\psi_n\}_{n\in\bbZ}$ span the range of $\calP_\varkappa$, this completes the proof. 
\end{proof}

\section*{Acknowledgements}
The authors are grateful to Yuri Lyubarskii and Alexander Sobolev for useful discussions.

\appendix 

\section{Ergodic operators -- a summary}

Here we very briefly recall some necessary statements from  \cite{Aiz-War,Pa-Fi:92} on spectral theory of ergodic operators. 

\subsection{General ergodic operators}\label{sec:A1}
Let $(\Omega,\calF,\bbP)$ be a probability space, where $\Omega$ is a set, $\calF$ is a sigma-algebra of measurable subsets on $\Omega$ and $\bbP$ is a probability measure defined on $\calF$. If a statement holds for $\bbP$-a.e. value of $\omega\in\Omega$, we will say that it holds \emph{almost surely} (or a.s.). We denote by $\bbE$ the operation of taking expectation with respect to $\bbP$.

Let $T$ be a measure-preserving automorphism of $\Omega$. We assume that $T$ is \emph{ergodic}, i.e. if $X\subset\Omega$ is measurable and $T(X)=X$, then $\bbP(X)=0$ or $\bbP(X)=1$. If a measurable function $f:\Omega\to\bbR$ is invariant under $T$, i.e. $f\circ T=f$, then $f$ is constant almost surely.

Let $\{f_\omega(n)\}_{n\in\bbZ}$ be a real-valued ergodic sequence, i.e. $\omega\mapsto f_\omega(n)$ is measurable for all $n$ and 
\[
f_{T\omega}(n)=f_\omega(n+1) \text{ for all $n$.}
\]
The expectation $\bbE\{f_\omega(n)\}$ (finite or infinite) is independent of $n$, and therefore it is customary to take $n=0$ and write $\bbE\{f_\omega(0)\}$ for this expectation. 

We state the Birkhoff--Khintchine ergodic theorem in the simplest form. 
\begin{proposition}\label{prp:A0}
Let $\{f_\omega(n)\}_{n\in\bbZ}$, be a real-valued ergodic sequence such that $\bbE\{\abs{f_\omega(0)}\}<\infty$. Then almost surely, 
\[
\lim_{N\to\infty}\frac1{2N+1}\sum_{n=-N}^N f_\omega(n)=\bbE\{f_\omega(0)\}.
\]
\end{proposition}
In other words, ``spatial average'' coincides with expectation. If $\{f_\omega(n)\}_{n\in\bbZ}$ is a sequence of i.i.d. random variables, this is simply the law of large numbers. 

Now let $\calH$ be a Hilbert space and let $A_\omega$, $\Omega\in\Omega$, be a measurable family of bounded self-adjoint operators in $\calH$. (Measurability means that for any $f,g\in\calH$, the inner product $\jap{A_\omega f,g}$ is a measurable function of $\omega\in\Omega$.) We refer to the family $A_\omega$ as a \emph{random operator} (we use the singular form, following the pattern of \emph{random variable}). We say that the random operator $A_\omega$ is \emph{ergodic}, if there exists a unitary operator $U$ in $\calH$ such that 
\begin{equation}
A_{T\omega}=UA_\omega U^*
\label{eq:A1}
\end{equation}
almost surely. This implies that for any Borel function $f$, we have 
\begin{equation}
f(A_{T\omega})=Uf(A_\omega) U^*,
\label{eq:A2}
\end{equation}
i.e. $f(A_\omega)$ is also ergodic. In particular, the spectral projection $\chi_\Delta(A_\omega)$ is an ergodic operator, where $\Delta\subset\bbR$ is any Borel set. 

\begin{proposition}\cite[Theorem~2.16]{Pa-Fi:92}\label{prp:A1}
Let $A_\omega$ be an a.s. bounded self-adjoint ergodic operator as above. Then the spectrum of $A_\omega$ is non-random. In other words, there exists a compact set $\Sigma\subset\bbR$ such that the spectrum of $A_\omega$ coincides with $\Sigma$ a.s. Moreover, the continuous, absolutely continuous, singular continuous and pure point components of the spectrum are nonrandom sets.
\end{proposition}

The key observation in the proof of the above proposition is extremely simple. Take any interval $\Delta\subset\bbR$ and consider the spectral projection $\chi_\Delta(A_\omega)$. By \eqref{eq:A2} (with $f=\chi_\Delta$), the function $d(\omega):=\Tr\chi_\Delta(A_\omega)$ is $T$-invariant, hence it is constant a.s. on $\Omega$. Thus, either $\Delta$ has a non-empty intersection with the spectrum of $A_\omega$ a.s. (if $d(\omega)>0$) or $\Delta$ is in the resolvent set of $A_\omega$ a.s.  (if $d(\omega)=0$). This shows that the spectrum is deterministic.

We note that in particular, the norm $\norm{A_\omega}$ is non-random. Thus, if $A_\omega$ is bounded almost surely, then it is bounded uniformly over $\omega$. 

\subsection{Ergodic operators in $\ell^2(\bbZ)$ and the IDS measure}\label{sec:A2}
Now suppose that our Hilbert space  is $\calH=\ell^2(\bbZ)$ and the operator $U$ in \eqref{eq:A1} is the shift operator:
\[
U\delta_n=\delta_{n-1}. 
\]
For every $\omega$, we consider the matrix of $A_\omega$ in the canonical basis $\{\delta_n\}$ of $\ell^2(\bbZ)$:
\[
a_\omega(n,m):=\jap{A_\omega\delta_m,\delta_n}. 
\]
By \eqref{eq:A1}, the diagonal $\{a_\omega(n,n)\}_{n\in\bbZ}$ is an ergodic sequence, i.e. satisfies
\[
a_{T\omega}(n,n)=a_\omega(n+1,n+1), 
\]
and in particular the expectation $\bbE\{a_\omega(n,n)\}$ is independent of $n$. It is customary to take $n=0$ and write $\bbE\{a_\omega(0,0)\}$.

Using this consideration, we define the probability measure 
\begin{equation}
\nu(\Delta)=\bbE\{\jap{\chi_\Delta(A_\omega)\delta_0,\delta_0}\}
\label{eq:A3}
\end{equation}
for a Borel set $\Delta\subset\bbR$. It is called the \emph{Integrated Density of States measure} of $A_\omega$. It is easy to prove \cite[Theorem~3.1]{Pa-Fi:92}  that the support of $\nu$ coincides with the deterministic spectrum of $A_\omega$. By the ergodic theorem, the average $\bbE$ over $\omega$ in \eqref{eq:A3} can be replaced by the average over $n$: 
\begin{equation}
\nu(\Delta)=\lim_{N\to\infty}\frac1{2N+1}\sum_{n=-N}^N\jap{\chi_\Delta(A_\omega)\delta_n,\delta_n}
\label{eq:A4}
\end{equation}
for all $\omega\in\Omega_\Delta$, where $\Omega_\Delta\subset\Omega$ is a set of full measure. In general, $\Omega_\Delta$ depends on $\Delta$ and it is not clear whether one can choose a set of full measure independently on $\Delta$ (as the set of all possible $\Delta$'s is uncountable). However, it is often sufficient to deal with an integrated version of \eqref{eq:A4}, viz. 
\begin{equation}
\int_{\bbR}f(\lambda)\dd\nu(\lambda)=\lim_{N\to\infty}\frac1{2N+1}\sum_{n=-N}^N\jap{f(A_\omega)\delta_n,\delta_n}
\label{eq:A5}
\end{equation}
for continuous $f$, which holds on a set of full measure independent on $f$ (the key consideration here is the separability of the space of continuous functions). 

\subsection{A Szeg\H{o} type theorem for the IDS measure}\label{sec:A3}
Let $1_N$ be the orthogonal projection in $\ell^2(\bbZ)$ onto the span of the $2N+1$ vectors $\{\delta_n\}_{n=-N}^N$ of the standard basis. The right-hand side of \eqref{eq:A5} can be written as 
\[
\lim_{N\to\infty}\frac1{2N+1}\Tr (1_Nf(A_\omega)1_N).
\]
It is often convenient to replace $\Tr (1_Nf(A_\omega)1_N)$ here by $\Tr f(A_\omega^{(N)})$, where $A_\omega^{(N)}$ is the compression of $A_\omega$ onto the $2N+1$-dimensional subspace:
\[
A_\omega^{(N)}=\{a_\omega(n,m)\}_{n,m=-N}^N. 
\]
The following theorem can be regarded as a stochastic analogue of the First Szeg\H{o} limit theorem (the latter theorem applies to convolution operators on $\ell^2(\bbZ)$). 

\begin{proposition}\cite[Theorem~4.11]{Pa-Fi:92}\label{prp:A2}
Let $A_\omega$ be an a.s. bounded self-adjoint ergodic operator on $\ell^2(\bbZ)$. Then for all continuous functions $f$ we have a.s. 
\[
\int_{\bbR}f(\lambda)\dd\nu(\lambda)=
\lim_{N\to\infty}
\frac1{2N+1}\Tr f(A_\omega^{(N)}).
\]
Moreover, if $\nu(\{\lambda\})=0$, then 
\[
\nu((\lambda,\infty))=
\lim_{N\to\infty}
\frac1{2N+1}\calN(\lambda,A_\omega^{(N)}).
\]
\end{proposition}
This proposition has an important corollary which allows one to apply variational estimates to the IDS measure. 
\begin{proposition}\cite[Theorem~4.14]{Pa-Fi:92}
\label{prp:A2a}
Let $A_\omega$ and $B_\omega$ be two uniformly bounded ergodic operators on $\ell^2(\bbZ)$, and let $\nu_A$ and $\nu_B$ be the corresponding IDS measures. Assume that $A_\omega\leq B_\omega$  (in the quadratic form sense) almost surely. Then for all $\lambda\in\bbR$, 
\[
\nu_A((\lambda,\infty))\leq \nu_B((\lambda,\infty)). 
\]
\end{proposition}

\subsection{Point masses of $\nu$}
Let, as above, $A_\omega$ be an ergodic operator in $\ell^2(\bbZ)$ with the IDS measure $\nu$. 
\begin{proposition}\cite[Theorems~2.11-2.12]{Pa-Fi:92}\label{prp:A3}
Let $\lambda\in\bbR$. 
\begin{enumerate}[\rm (i)]
\item
If $\nu(\{\lambda\})>0$, then $\lambda$ is an eigenvalue of $A_\omega$ of infinite multiplicity almost surely. 
\item
If $\nu(\{\lambda\})=0$, then $\lambda$ is not an eigenvalue of $A_\omega$ almost surely.
\end{enumerate}
In particular, if $\lambda$ is an eigenvalue of $A_\omega$ almost surely, then it has infinite multiplicity almost surely.
\end{proposition}

\end{document}